\documentclass[11pt,reqno]{amsart} 
\usepackage[utf8]{inputenc}
\usepackage{amsfonts,amsmath,amsthm,amssymb,fullpage}
\usepackage[justification=centering]{caption}
\usepackage{color}
\usepackage{enumitem}  
\usepackage{hyperref}
\usepackage{mathtools}
\usepackage{xcolor}
\usepackage{thm-restate}
\usepackage{thmtools}
\usepackage[backend=biber,sorting=nty,style=numeric,maxbibnames=99]{biblatex}
\makeatletter
\newcommand{\addresseshere}{%
  \enddoc@text\let\enddoc@text\relax
}
\makeatother

\DeclareMathOperator*{\argmax}{arg\,max}
\DeclareMathOperator*{\argmin}{arg\,min}

\newtheorem{theorem}{Theorem}[section]

\newtheorem{lemma}[theorem]{Lemma}
\newtheorem{corollary}[theorem]{Corollary}

\newtheorem{mainthm}{Theorem}

\theoremstyle{definition}
\newtheorem{remark}[theorem]{Remark}
\newtheorem{example}[theorem]{Example}
\newtheorem{definition}[theorem]{Definition}

\newtheorem{openproblem}{Open Problem}
\newtheorem{claim}{Claim}[theorem]
\newtheorem{case}{Case}[theorem]
\renewcommand*{\b}{\mathbf{b}}
\newcommand*{\y}{\mathbf{y}}
\newcommand*{\x}{\mathbf{x}}
\newcommand*{\z}{\mathbf{z}}

\newcommand*{\PF}{\mathsf{PF}}
\newcommand*{\PFND}{\mathsf{PF}^{\uparrow}}
\newcommand*{\N}{\mathbb{N}}
\newcommand*{\PA}{\mathsf{PA}}
\newcommand*{\PAINV}{\mathsf{PA}^{\mathrm{inv}}}
\newcommand*{\PAINVND}{\mathsf{PA}^{\mathrm{inv},\uparrow}}
\newcommand*{\W}{\mathcal{W}}
\newcommand*{\floor}[1]{\left\lfloor #1 \right\rfloor}
\title{
On the Structure of Permutation Invariant Parking
}
\author[Chen]{Douglas M. Chen}
\address[D.~M.~Chen]{Department of Mathematics, Johns Hopkins University, Baltimore, MD 21218}
\email{\textcolor{blue}{\href{mailto:dchen101@jhu.edu}{dchen101@jhu.edu}}}

\begin{document}
\begin{abstract}
We continue the study of parking assortments, a generalization of parking functions introduced by Chen, Harris, Mart\'{i}nez, Pab\'{o}n-Cancel, and Sargent. 
Given $n \in \N$ cars of lengths $\y=(y_1,y_2,\dots,y_n) \in \N^n$, our focuses are the sets $\PAINV_n(\y)$ and $\PAINVND_n(\y)$ of permutation invariant (resp. nondecreasing) parking assortments for $\y$.
For $\x=(x_1,x_2,\dots,x_n) \in \PAINV_n(\y)$, we introduce the \emph{degree} of $\x$, the number of non-$1$ entries of $\x$, and the \emph{characteristic} $\chi(\y)$ of $\y$, the greatest degree across all $\z \in \PAINV_n(\y)$.
We establish direct necessary conditions for $\y$ with $\chi(\y)=0$ and a simple characterization for $\y$ with $\chi(\y)=n-1$.
In the process, we derive a closed form for $\PAINV_n(\y)$ and an enumeration of $|\PAINV_n(\y)|$ using properties of the Pitman-Stanley polytope, where $\chi(\y)=n-1$.
% Next, for any $\y \in \N^n$, we prove closure and embedding properties of $\PAINV_n(\y)$ and apply these to compute the image of the degree and show a monotonicity property of the characteristic.
Next, for any $\y \in \N^n$, we prove that $\PAINV_n(\y)$ is closed under the replacement of any of its elements' entries by a $1$, and given $\y^+ \in \N^{n+1}$, where $\y$ is the prefix of $\y^+$, there is an embedding of $\PAINVND_n(\y)$ into $\PAINVND_{n+1}(\y^+)$.
We apply these results to study the degree as a function and the characteristic under sequences of successive prefix length vectors.
We then examine the \emph{invariant solution set} $\W(\y)\coloneqq \{ w \in \N:(1^{n-1},w) \in \PAINV_n(\y) \}$.
We obtain tight upper bounds of $\W(\y)$ and prove that for any $n \in \N$, we have $|\W(\y)|\leq 2^{n-1}$, providing constraints on the subsequence sums of $\y$ for equality to hold. 
% Finally, we show that if $\x \in \PAINV_n(\y)$ is nondecreasing, then $\x \in \{ 1 \}^{n-\chi(\y)} \times \W(\y)^{\chi(\y)}$, which implies a new upper bound on the number of elements of $\PAINV_n(\y)$ up to permutation.
Finally, we show that if $\x \in \PAINVND_n(\y)$, then $\x \in \{ 1 \}^{n-\chi(\y)} \times \W(\y)^{\chi(\y)}$, which implies a new upper bound on $|\PAINVND_n(\y)|$.
Our results generalize several theorems by Chen, Harris, Mart\'{i}nez, Pab\'{o}n-Cancel, and Sargent.
\end{abstract}
\maketitle

\section{Introduction}
\label{section: introduction}
Parking assortments were introduced recently in ~\cite{icermpaper2023inv} as a generalization of parking functions, classic combinatorial objects that arose in the 1960s from the study of hash functions and linear probing ~\cite{konheim1966occupancy}.
% \begin{notation}
%     We will adopt the following conventions:
%     \begin{itemize}
%         \item $\N_0 \coloneqq \N \cup \{ 0 \}$.
%         % \item $[n]_0 \coloneqq [n] \cup \{ 0 \}$.
%         % \item $\underbrace{(c,c,\dots,c)}_{n \ c\text{s}} \coloneqq (c^n)$.
%     \end{itemize}
% \end{notation}
There are two well-known equivalent definitions of parking functions.
The first can be described via a ``parking experiment."
\begin{definition}
\label{defn: pfexperiment}
    Consider a one-way street with $n \in \N\coloneqq \{ 1,2,3,\dots \}$ parking spots. 
    There are $n \in \N$ cars of unit length waiting to enter the street.
    For each $i\in [n]\coloneqq \{ 1,2,\dots,n \}$, car $i$ prefers a spot $x_i \in [n]$ of the parking lot, so it drives up to $x_i$ and parks if $x_i$ is unoccupied; otherwise, it parks in the next available spot (if it exists). 
    We say that $\x\coloneqq (x_1,x_2,\dots,x_n) \in [n]^n$ is a \emph{parking function of length $n$} if every car can park following their respective preference in $\x$.
\end{definition}
The second is via a property of $(x_{(1)},x_{(2)},\dots,x_{(n)})$, which is the nondecreasing rearrangement of the entries of $\x$.
% \begin{notation}
%     Let $\x=(x_1,x_2,\dots,x_n) \in \N^n$ and let $(x_{(1)},x_{(2)},\dots,x_{(n)})$ be the nondecreasing rearrangement of the entries of $\x$. 
%     Then $x_{(i)}$ is the \emph{$i$th order statistic of $\x$}, i.e. the $i$th smallest entry of $\x$. 
% \end{notation}
\begin{definition}
\label{defn: pforderstat}
    We say that $\x=(x_1,x_2,\dots,x_n) \in [n]^n$ is a \emph{parking function of length $n$} if $x_{(i)}\leq i$ for all $i \in [n]$.
\end{definition}
Perhaps the most celebrated result concerning parking functions is the fact that
\[ |\PF_n|=(n+1)^{n-1}, \]
where $\PF_n$ denotes the set of parking functions of length $n$ (cf. Lemma 1 in \cite{konheim1966occupancy}).
Such a count allows one to see that parking functions are in bijection with several notable combinatorial objects including labeled trees and the Shi hyperplane arrangement \cite{yan2015parking}. 
Explicitly constructing and investigating these bijections reveals many illuminating combinatorial properties of parking functions.
As an example, the sequence 
\[ (c_1,c_2,\dots,c_{n-1}), \quad \text{where} \quad c_i\coloneqq (x_{i+1}-x_i) \pmod{n+1} \quad \forall i \in [n-1], \]
of successive differences modulo $n+1$ of an $\x \in \PF_n$ yields the Prüfer code of a unique labeled tree on $n+1$ vertices \cite{yan2015parking}.
Moreover, let $\mathbf{u}=(u_1,u_2,\dots,u_n) \in \N^n$ be nondecreasing, and consider the \emph{Pitman-Stanley polytope} 
\[ \Pi_n(\mathbf{u})\coloneqq \left\{ (p_1,p_2,\dots,p_n) \in \mathbb{R}_{\geq 0}^n:\sum_{i=1}^{j}p_i\leq \sum_{i=1}^{j}u_i \quad \forall j\in [n] \right\}. \] 
With $\mathbf{u}=(1,2,\dots,n)$, we have $n!V_n(\mathbf{u})=|\PF_n|$, where $V_n(\mathbf{u})$ denotes the $n$-dimensional volume of $\Pi_n(\mathbf{u})$.
More generally, $n!V_n(\mathbf{u})$ is the number of $\x=(x_1,x_2,\dots,x_n) \in \N^n$ such that $x_{(i)}\leq u_i$ for all $i \in [n]$ (cf. Theorem 11 in ~\cite{pitmanstanley2002polytope}); such $\x$ generalize Definition \ref{defn: pforderstat} and are aptly known as \emph{$\mathbf{u}$-parking functions of length $n$}.
For a detailed treatment on the combinatorial theory of parking functions, see Yan \cite{yan2015parking}.

Given these connections and results, parking functions have been an active research topic, appearing in a diverse array of contexts such as pattern avoidance in permutations ~\cite{harris2023outcome}, convex geometry ~\cite{amanbayeva2022convex}, polyhedral combinatorics ~\cite{benedetti2019combinatorial}, partially ordered sets ~\cite{elder2023boolean}, impartial games ~\cite{ji2021brussels}, and of course, the combinatorics of its variations and generalizations ~\cite{colmenarejo2021counting, franks2023counting}.

Now, we define the combinatorial object of interest in this work: parking assortments.
This will involve a natural extension to the parking experiment described in Definition \ref{defn: pfexperiment}.
\begin{definition}
    Consider a one-way street with $m \in \N$ parking spots. 
    There are $n \in \N$ cars waiting to enter the street, and they have lengths $\y\coloneqq (y_1,y_2,\dots,y_n) \in \N^n$, where $m=\sum_{i=1}^{n}y_i$. 
    For each $i\in [n]$, car $i$ (with length $y_i$) prefers a spot $x_i \in [m]$ of the parking lot, so it drives up to $x_i$ and parks if spots $x_i,x_i+1,\dots,x_i+y_i-1$ are unoccupied; otherwise, it parks in the next $y_i$ contiguously available spots (if they exist). 
    We say that $\x\coloneqq (x_1,x_2,\dots,x_n) \in [m]^n$ is a \emph{parking assortment} for $\y$ if every car can park following their respective preference in $\x$.
\end{definition}
\begin{example}
    See Figure $\ref{figure: parkingexample}$, where $\y=(3,4,2)$ and $\x=(5,1,6)$.
    The parking experiment proceeds as follows.
    First, car $1$, with length $3$ and preference $5$, occupies spots $5$, $6$, and $7$.
    Then, car $2$, with length $4$ and preference $1$, occupies spots $1$, $2$, $3$, and $4$ because they are all still available.
    Lastly, car $3$, with length $2$ and preference $6$, occupies spots $8$ and $9$ because spots $6$ and $7$ have been occupied (namely by car $1$), and spots $8$ and $9$ are the next $2$ contiguously available spots. 
    \begin{figure}
        \centering
        \includegraphics[width=7.425cm]{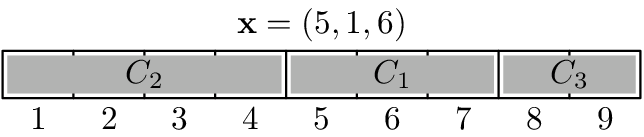}
        \caption{If $\y=(3,4,2)$, then $\x \in \PA_3(\y)$.}
        \label{figure: parkingexample}
    \end{figure}
\end{example}
% \begin{notation}
%     Let $\y\in \N^n$. We use $\PA_n(\y)$ to denote the set of parking assortments for $\y$.
% \end{notation}
Note that if $\y=(1^n)\coloneqq (1,1,\dots,1) \in \N^n$, then the set $\PA_n(\y)$ of parking assortments for $\y$ is precisely $\PF_n$. 
From Definition \ref{defn: pforderstat}, we see that this set has the peculiar property that for any $\x=(x_1,x_2,\dots,x_n) \in \PA_n(\y)$, any permutation of its entries $\x'=(x_{\sigma(1)},x_{\sigma(2)},\dots,x_{\sigma(n)})$, where $\sigma \in \mathfrak{S}_n$, is also in $\PA_n(\y)$. 
This permutation invariance of parking functions is a crucial property used to study $\PF_n$ (cf. \S 1.1 in ~\cite{yan2015parking}).
However, this is not true for general $\y$; if $\y=(1,2,2)$, then $\x=(1,1,2) \notin \PAINV_3(\y)$ since $(2,1,1) \notin \PA_3(\y)$, as detailed in ~\cite{icermpaper2023inv}. 
This motivates the problem of determining what $\x \in \PA_n(\y)$ have this property and hence the definition below. 
\begin{definition}
    Let $\y\in \N^n$. 
    We say $\x \in \PA_n(\y)$ is a \emph{(permutation) invariant parking assortment for} $\y$ if any permutation of its entries is also in $\PA_n(\y)$. 
\end{definition}
% \begin{notation}
%     Let $\y \in \N^n$. 
%     Similar to the above convention, we use $\PAINV_n(\y)$ to denote the set of invariant parking assortments for $\y$, and $\PAINVND_n(\y) \subseteq \PAINV_n(\y)$ will denote its subset of nondecreasing elements.
%     Note that $\PAINV_n(\y)$ and $\PAINVND_n(\y)$ are equivalent up to permutation. 
% \end{notation}
In this work, we are interested in the set $\PAINV_n(\y)$ of invariant parking assortments for $\y$.
Alternatively, we may study the set $\PAINVND_n(\y)$ of nondecreasing invariant parking assortments for $\y$, as it is equivalent to $\PAINV_n(\y)$ up to permutation.
\begin{example}
    Let $\y=(7,4,6) \in \N^3$.
    Then $\PAINVND_3(\y)=\{ (1^3),(1^2,5) \}$.
    Indeed, $(1^3) \in \PA_3(\y)$, and $(1^2,5),(1,5,1),(5,1^2) \in \PA_3(\y)$; in other words, all permutations of the entries of $(1^3)$ and $(1^2,5)$ are also parking assortments, so $\{ (1^3),(1^2,5) \} \subseteq \PAINVND_3(\y)$.
    One can also check that no other elements of $[7+4+6]^3=[17]^3$ have this property, which yields the reverse inclusion.
\end{example}

Determining $\PAINV_n(\y)$ is generally an arduous task due to the problem of checking each of the elements of $[m]^n$ for invariance.
The main motivator behind the notions introduced in this work is to attempt to reduce this search space considerably.

A basic approach to narrowing the possibilities is to ``decrease the exponent" by noticing that by definition, parking assortments have at least one entry that is equal to $1$; otherwise, parking spot $1$ is never occupied by any car.
This reduces the search space to a size of $m^{n-1}$.
We can take this idea a step further and specifically study the number of entries of invariant parking assortments that are not equal to $1$, a fundamental property we will refer to as their \emph{degree}.
Heuristically, one can expect that most of the entries of an $\x \in \PAINV_n(\y)$ are $1$, and so many of the results presented here will concern and analyze the maximum degree across all $\z \in \PAINV_n(\y)$, which we call the \emph{characteristic} of $\y$.
If there is a straightforward way to bound the characteristic, then this can reduce our search space by more factors of $m$.
\begin{definition}
\label{defn: degandchar}
    Let $\y \in \N^n$. 
    For any $\x=(x_1,x_2,\dots,x_n) \in \PAINV_n(\y)$, the \emph{degree} of $\x$ is given by
    \[ \deg \x\coloneqq |\{ i \in [n]:x_i\neq 1 \}|. \]
    Moreover, the \emph{characteristic} of $\y$ is given by
    \[ \chi(\y)\coloneqq \max_{\z \in \PAINV_n(\y)} \deg \z. \]
    As a special case, note that $\chi(\y)=0$ is equivalent to $\y$ being \emph{minimally invariant}, which is when $\PAINV_n(\y)=\{ (1^n) \}$ \cite{icermpaper2023inv}. 
    Furthermore, we always have $\chi(\y)\in [n-1]_0\coloneqq [n-1] \cup \{ 0 \}$.
\end{definition}
% \begin{notation}
%     Let $\mathbf{v}=(v_1,v_2,\dots,v_n) \in \N^n$, $w \in \N$, and $i \in [n]$.
%     The following are various list operations that will be used:
%     \begin{itemize}
%         \item $\mathbf{v}_{\widehat{i}}\coloneqq (v_1,v_2,\dots,v_{i-1},v_{i+1},\dots,v_n) \in \N^{n-1}$.
%         \item $\mathbf{v}_{\vert_i} \coloneqq (v_1,v_2,\dots,v_i) \in \N^i$.
%         \item $(\mathbf{v},w)\coloneqq (v_1,v_2,\dots,v_n,w) \in \N^{n+1}$.
%     \end{itemize}
% \end{notation}

In Section \ref{section: almostconstant}, we study $\y$ of minimal and maximal characteristic. 
We first establish a direct necessary condition for $\y$ to have minimal characteristic.
\begin{mainthm}
\label{mainthm: neccminchar}
    Let $\y=(y_1,y_2,\dots,y_n) \in \N^n$. 
    If $\y$ is minimally invariant, then 
    \[ y_1<\min(y_2,y_3,\dots,y_n) \quad \text{and} \quad y_2\neq \sum_{j \in [n] \setminus \{ 2 \}}y_j. \]
\end{mainthm}
Moreover, we work on decreasing the exponent by considering $\PAINV_n(\y)$ for $\y=(b,a^{n-1})$ and $n\geq 2$.
In particular, we obtain a simple necessary and sufficient condition for $\y \in \N^n$ to have maximal characteristic, which provides an easy way to reduce the search space by at least one more factor of $m$, as we need only check for invariant parking assortments of degree $n-1$ when this condition is satisfied.
\begin{mainthm}
\label{mainthm: maxchar}
    Let $\y=(y_1,y_2,\dots,y_n)\in \N^n$.
    Then $\chi(\y)=n-1$ if and only if
    \begin{equation} 
    \label{eqn: equivmaxchar}
        y_1\geq y_2 \quad \text{and} \quad y_2=y_3=\cdots=y_n.
    \end{equation}
\end{mainthm}
Additionally, we obtain an explicit description of $\PAINV_n(\y)$ and corresponding enumerative results.
The enumerative formulas for \ref{mainthm: almostconstant}\ref{mainthm: almostconstantirreg} are deduced via two results: an enumerative result related to the theory of the Pitman-Stanley polytope and empirical distributions \cite{pitmanstanley2002polytope} and a recursive formula for Catalan's triangle \cite{bailey1996catalan}.
\begin{mainthm}
\label{mainthm: almostconstant}
    Let $\y=(b,a^{n-1}) \in \N^n$, where $n\geq 2$.
    \begin{enumerate}[label=(\alph*)]
        \item\label{mainthm: almostconstantreg} If $a \mid b$ or $b>(n-1)a$, then $\x \in \PAINV_n(\y)$ if and only if
        \[ x_{(i)} \in \{ 1+(k-1)a:k\in [i] \} \quad \forall i \in [n]. \]
        Moreover,
        \[ |\PAINV_n(\y)|=(n+1)^{n-1} \quad \text{and} \quad |\PAINVND_n(\y)|=\frac{1}{n+1}\binom{2n}{n}. \]
        \item\label{mainthm: almostconstantirreg} Otherwise, if $a \nmid b$ and $b<(n-1)a$, then $\x \in \PAINV_n(\y)$ if and only if
        \[ x_{(i)} \in \begin{cases}
            \{ 1+(k-1)a:k\in [i] \} & \forall i\in\left[\floor{\frac{b}{a}}+1\right] \\
            \left\{ 1+(k-1)a:k \in \left[ \floor{\frac{b}{a}}+1 \right] \right\} & \text{otherwise}.
        \end{cases}
        \]
        Moreover,
        \begin{align*}
            |\PAINV_n(\y)|&=\sum_{j=0}^{n-\floor{b/a}-1}(-1)^j\binom{n}{j}\left( n-\floor{\frac{b}{a}}-1 \right)^j(n-j+1)^{n-j-1} \quad \text{and} \\
            |\PAINVND_n(\y)|&=\frac{n-\floor{b/a}+1}{n+1}\binom{n+\floor{b/a}}{\floor{b/a}}.
        \end{align*}
        % \begin{align*} |\PAINV_n(\y)|&=\begin{cases}
        %     (n+1)^{n-1} & \text{if $a\mid b$ or $b>(n-1)a$} \\[5pt]
        %     \sum_{j=0}^{n-\floor{b/a}-1}(-1)^j\binom{n}{j}\left( n-\floor{\frac{b}{a}}-1 \right)^j(n-j+1)^{n-j-1} & \text{otherwise}
        %     \end{cases} \\
        %     |\PAINVND_n(\y)|&=\begin{cases}
        %         \frac{1}{n+1}\binom{2n}{n} & \text{if $a\mid b$ or $b>(n-1)a$} \\[5pt]
        %         \frac{n-\floor{b/a}+1}{n+1}\binom{n+\floor{b/a}}{\floor{b/a}} & \text{otherwise}.
        %     \end{cases}
        % \end{align*}
    \end{enumerate}
\end{mainthm}

In Section \ref{section: deg and char}, we 
% first continue in the same vein by considering similar problems when the characteristic is not maximal and then shift our focus to 
study the structure of $\PAINV_n(\y)$ for any $\y\in \N^n$, which implies special properties of the degree and characteristic.
Namely, we prove that $\PAINV_n(\y)$ is closed under the operation of replacing any of its elements' entries with a $1$.
We may use this property to compute the image of the degree as a function, reveal an embedding property of the set of invariant parking assortments, and obtain a bound for $\chi(\y^+)$ in terms of $\chi(\y)$, where $\y$ is the prefix of $\y^+ \in \N^{n+1}$.

To state these results precisely, we need the following list operations: for $\mathbf{v}=(v_1,v_2,\dots,v_n) \in \N^n$, $k \in \N$, and $i \in [n]$, let $\mathbf{v}_{\widehat{i}}\coloneqq (v_1,v_2,\dots,v_{i-1},v_{i+1},\dots,v_n) \in \N^{n-1}$, $\mathbf{v}_{\vert_i} \coloneqq (v_1,v_2,\dots,v_i) \in \N^i$, and $(1^k,\mathbf{v})\coloneqq (1^k,v_1,v_2,\dots,v_n) \in \N^{n+k}$.
% $\deg: \PAINV_n(\y) \to [n-1]_0$ has image $[\chi(\y)]_0$, and given $\y^+\coloneqq (\y,y_{n+1}) \in \N^{n+1}$, we have an embedding $\PAINVND_n(\y) \hookrightarrow \PAINVND_{n+1}(\y^+)$, which we use to prove that $\chi(\y^+)$ is either $\chi(\y)$ or $\chi(\y)+1$.

The inspiration for Theorem \ref{mainthm: degchar}\ref{mainthm: degchar closure} is due to~\cite{personalcommuncation}; we present an independent proof in Section \ref{section: deg and char}.
\begin{mainthm}
\label{mainthm: degchar}
    Let $\y=(y_1,y_2,\dots,y_n) \in \N^n$ and $\y^+=(\y,y_{n+1}) \in \N^{n+1}$.
    \begin{enumerate}[label=(\alph*)]
        \item\label{mainthm: degchar closure} If $(1^{n-d},\mathbf{w}) \in \PAINV_n(\y)$, where $\mathbf{w} \in \N^d_{>1}$, then $(1^{n-d+1},\mathbf{w}_{\widehat{i}}) \in \PAINV_n(\y)$ for all $i \in [d]$.
        \item\label{mainthm: degchar image} The image of $\deg: \PAINV_n(\y) \to [n-1]_0$ is $[\chi(\y)]_0$.
        In other words, for each $d \in [\chi(\y)]_0$, there exists $\x \in \PAINV_n(\y)$ with $\deg \x=d$.
        \item\label{mainthm: degchar embedding} If $\x \in \PAINV_n(\y)$, then $(1,\x) \in \PAINV_{n+1}(\y^+)$.
        In particular, we have the embedding \footnote{That is, $\PAINVND_n(\y)$ is included in $\PAINVND_{n+1}(\y^+)$ up to inserting a $1$ at the start of each of its elements.}
        \[ \eta: \begin{cases}
            \PAINVND_n(\y) & \hookrightarrow \PAINVND_{n+1}(\y^+) \\
            \x & \mapsto (1,\x).
            \end{cases} 
        \]
        \item\label{mainthm: degchar prefix} If $\chi(\y)=\alpha$, then $\chi(\y^+) \in \{ \alpha,\alpha+1 \}$.
    \end{enumerate}
\end{mainthm}

% We prove that if $\x \in \PAINV_n(\y)$, and we form $\Tilde{\x}$ by replacing any of the entries of $\x$ with a $1$, then $\Tilde{\x} \in \PAINV_n(\y)$, 
% As a consequence, we show that given $\y\in \N^n$ and $\y^+\coloneqq (\y,y_{n+1}) \in \N^{n+1}$, we have an embedding $\PAINVND_n(\y) \hookrightarrow \PAINVND_{n+1}(\y^+)$, which we use to prove that $\chi(\y^+)$ is either $\chi(\y)$ or $\chi(\y)+1$.
Note that applying Theorem \ref{mainthm: degchar}\ref{mainthm: degchar prefix} repeatedly implies that if $\chi(\y_{\vert_k})=\alpha$ for some $k<n$, then $\chi(\y) \in [\alpha,\alpha+n-k]$, which helps decrease the exponent, as discussed previously.
Moreover, the characteristic is monotonically increasing with respect to sequences $\{ \y^{(j)} \}_{j \in \N}$ of length vectors, where each $\y^{(j)}$ is the prefix of $\y^{(j+1)}$.

On the flip side, Theorem \ref{mainthm: degchar}\ref{mainthm: degchar image} implies that if the characteristic is known, then for every $d\in [\chi(\y)]_0$, each set $\{ 1 \}^{n-d} \times [m]^d$, which consists of candidates for $\x \in \PAINV_n(\y)$ with $\deg \x=d$, indeed contains at least one such $\x$.
Consequently, we must check each degree in $[\chi(\y)]_0$ to find invariant parking assortments.
Thus, we cannot take the approach of decreasing the exponent further.
% expresses the notion that once the characteristic is known, there is, in a sense, no way to further decrease the exponent, as for every $p\in [\chi(\y)]_0$, we know that $\{ 1 \}^{n-p} \times [m]^p$ contains an invariant parking assortment.
This motivates the need to consider the 
% other natural search space reduction
analogous approach of ``decreasing the base," where we eliminate elements of $[m]$ that cannot be an entry of an invariant parking assortment for $\y$.
Intuitively, such elements are either too large or not compatible with partial sums of the entries of $\y$.
The former intuition is made precise in Proposition 2.2 in~\cite{icermpaper2023inv}. 
The latter intuition is the basis for Section \ref{section: sdlengths}, where we also sharpen the former intuition.
% Here, we consider $\W(\y)\coloneqq \{ w \in \N:(1^{n-1},w) \in \PAINV_n(\y) \}$, and we obtain an upper bound for $|\W(\y)|$, among other results.
\begin{mainthm}
\label{mainthm: invsolset}
    Let $\y=(y_1,y_2,\dots,y_n) \in \N^n$, and define $\W(\y)\coloneqq \{ w \in \N:(1^{n-1},w) \in \PAINV_n(\y) \}$.
    \begin{enumerate}[label=(\alph*)]
        \item \label{mainthm: invsolset nonconstant} If $\y$ is non-constant and $w \in \W(\y)$, then $w\leq \sum_{j=1}^{n-1}y_j$.
        \item\label{mainthm: invsolset size bound} We have
        \[ \W(\y)\subseteq \{ 1+\b^\top\y:\b \in \{ 0 \} \times \{ 0,1 \}^{n-1} \} \quad \text{and} \quad |\mathcal{W}(\y)|\leq 2^{n-1}. \]
        \item\label{mainthm: invsolset supdec} If
        \begin{equation}
        \label{eqn: supdec}
            y_j>\sum_{i=j+1}^{n}y_i \quad \forall j \in [n-1],
        \end{equation}
        then equality is achieved in \ref{mainthm: invsolset size bound}.
        Moreover,
        \[ |\PAINV_n(\y)|=2^{n-1}n-n+1 \quad \text{and} \quad |\PAINVND_{n}(\y)|=2^{n-1}. \]
        \item\label{mainthm: invsolset neccequality} If equality is achieved in \ref{mainthm: invsolset size bound}, then
        \begin{equation} 
        \label{eqn: almostsupdec}
            y_1\geq y_2 \quad \text{and} \quad y_j>\sum_{i=j+1}^{n}y_i \quad \forall j \in [n-1] \setminus \{ 1 \}. 
        \end{equation}
    \end{enumerate}
\end{mainthm}
% where we demonstrate that $\W(\y)\subseteq \{ 1+\b^\top\y:\b \in \{ 0 \} \times \{ 0,1 \}^{n-1} \}$ in the process and study a family of certain strictly decreasing $\y$ to show that this bound is sharp.
% For such $\y$, we obtain the enumerative results
% \[ |\PAINV_n(\y)|=2^{n-1}n-n+1 \quad \text{and} \quad |\PAINVND_{n}(\y)|=2^{n-1}. \]
Lastly, we combine the approaches of decreasing the exponent and decreasing the base to show the following inclusion and upper bound.
\begin{mainthm}
\label{mainthm: PAINVNDbound}
    Let $\y \in \N^n$.
    Then $\PAINVND_n(\y) \subseteq \{ 1 \}^{n-\chi(\y)}\times \W(\y)^{\chi(\y)}$.
    In particular, 
    \[ |\PAINVND_n(\y)|\leq \binom{2^{n-1}+n-2}{n-1}. \] 
\end{mainthm}
We conclude in Section \ref{section: openproblems} with several open problems to spur future study.
\section{Length Vectors of Minimal and Maximal Characteristic}
% Almost Constant Car Lengths and Their Characteristics
\label{section: almostconstant}
In this section, we study the inverse problem of determining $\y \in \N^n$ of a given characteristic $\alpha$.
Doing so can still allow us to recover $\PAINV_n(\y)$; notably, this is immediate for $\alpha=0$, and it turns out that if $\alpha=n-1$, then we can explicitly compute $\PAINV_n(\y)$.
In \S\ref{subsec: minchar}, we obtain a direct necessary condition for $\y$ to be of minimal characteristic, thereby establishing Theorem \ref{mainthm: neccminchar}.
In contrast, \S\ref{subsec: maxchar} proves a simple equivalent condition for $\y$ to be of maximal characteristic and provides a closed form of $\PAINV_n(\y)$ for all such $\y$.
By establishing a bijection between certain $\mathbf{u}$-parking functions and elements of $\PAINV_n(\y)$, we utilize a result of Pitman and Stanley to enumerate $|\PAINV_n(\y)|$; we deduce a recursive formula to enumerate $|\PAINVND_n(\y)|$.
Altogether, \S\ref{subsec: maxchar} proves Theorems \ref{mainthm: maxchar} and \ref{mainthm: almostconstant}.
\subsection{Minimal Characteristic}
\label{subsec: minchar}
Recall the following result for minimally invariant $\y \in \N^3$.
\begin{theorem}[Theorem 5.3 in ~\cite{icermpaper2023inv}]
\label{theorem: mi3}
    Let $\y=(y_1,y_2,y_3)\in \N^3$. 
    Then, $\y$ is minimally invariant if and only if $y_1<\min(y_2,y_3)$ and $y_2\neq y_1+y_3$.
\end{theorem}
Thus, one might ask whether this can be generalized for $n>3$.
It turns out that the natural extension of this result is necessary for all $n \in \N$ but not sufficient.
We will first prove the necessity, which is Theorem \ref{mainthm: neccminchar}.
The failure of the sufficiency will be easily established after we prove Theorem \ref{mainthm: degchar}\ref{mainthm: degchar closure} in Section \ref{section: deg and char}.
% \begin{theorem}
% \label{theorem: minecessary}
%     Let $\y=(y_1,y_2,\dots,y_n) \in \N^n$. 
%     If $\chi(\y)=0$, then 
%     \[ y_1<\min(y_2,y_3,\dots,y_n) \quad \text{and} \quad y_2\neq \sum_{j \in [n] \setminus \{ 2 \}}y_j. \]
% \end{theorem}
\begin{proof}[Proof of Theorem \ref{mainthm: neccminchar}]
    We consider the contrapositive. 
    
    Assume first that $y_1\geq \min(y_2,y_3,\dots,y_n)$. 
    We claim $(1^{n-1},1+\min \y) \in \PAINV_n(\y)$.
    As $y_1\geq \min \y$, we have $1+\min \y\leq 1+\sum_{j=1}^{i}y_j$ for all $i \in [n-1]$, so $(1^i,1+\min \y) \in \PA_n(\y_{\vert_{i+1}})$ by Lemma \ref{lemma: nondecreasing}. 
    It follows that $(1^i,1+\min \y,1^{n-i-1}) \in \PA_n(\y)$.
    Indeed, under its parking experiment, the last $n-i-1$ cars, all with preference $1$, are forced to successively fill $[1+\sum_{j=1}^{i+1}y_j,m]$, as the first $i+1$ cars already occupy $[1,\sum_{j=1}^{i+1}y_j]$ since $(1^i,1+\min \y) \in \PA_n(\y_{\vert_{i+1}})$.
    Lastly, to show that $(1+\min \y,1^{n-1}) \in \PA_n(\y)$, note that under its parking experiment, the first car leaves $[1,\min \y]$ unoccupied, and these spots can only be filled by one car. 
    The remaining cars, all with preference $1$, either park starting at the first empty spot after $[1,\min\y]$ (if its length exceeds $\min\y$) or fill $[1,\min \y]$ (if its length is precisely $\min \y$ and no other cars with length $\min \y$ have attempted to park). 
    For the latter possibility, there must exist a smallest index $j>1$ such that $j=\argmin \y$ by assumption.
    Thus, $[1,\min\y]$ will be filled by car $j$, after which point $[1,s]$ is occupied for some $s \in \N$, so cars $j+1,j+2,\dots,n$ will successively fill $[s+1,m]$. 
    Hence, $(1^{n-1},1+\min\y) \in \PAINV_n(\y)$, as needed.

    Now, assume $y_2=\sum_{j \in [n] \setminus \{ 2 \}}y_j$. 
    We claim that $\x=(1^{n-1},1+y_2) \in \PAINV_n(\y)$. 
    To see this, note first that $(1+y_2,1^{n-1}) \in \PA_n(\y)$, as under its parking experiment, the first two cars will occupy $[1,y_1+y_2]$, and the remaining cars all have preference $1$, so they will park successively in $[1+y_1+y_2,m]$. 
    Moreover, $(1,1+y_2,1^{n-2}) \in \PA_n(\y)$ because under its parking experiment, the first car parks in $[1,y_1]$, while the second car parks in $[1+y_2,2y_2]=[m-y_2+1,m]$ since $m=\sum_{i=1}^{n}y_i=2y_2$. 
    Since $\sum_{j=3}^{n}y_j=m-y_1-y_2=y_2-y_1$ and the remaining cars all have preference $1$, they will park successively in $[1+y_1,m-y_2]$. 
    Otherwise, let $\x'=(x_1',x_2',\dots,x_n')$ be a permutation of $\x$ such that $x_j'=1+y_2$, where $j>2$, and consider the parking experiment for $\x'$. 
    Cars $1,2,\dots,j-1$ all have preference $1$, so they will occupy $[1,\sum_{i=1}^{j-1} y_i]$, where $|[1,\sum_{i=1}^{j-1} y_i]|>y_2$. 
    As car $j$ has preference $1+y_2$, it will occupy $[1+\sum_{i=1}^{j-1}y_i,\sum_{i=1}^{j}y_i]$. 
    Lastly, any remaining cars again all have preference $1$, so they will park successively in $[1+\sum_{i=1}^{j}y_i,m]$.
    Hence, $(1^{n-1},1+y_2) \in \PAINV_n(\y)$, as needed.

    Therefore, if $y_1\geq \min(y_2,y_3,\dots,y_n)$ or $y_2=\sum_{j \in [n] \setminus \{ 2 \}}y_j$, then $\chi(\y)\geq 1$, as desired.
\end{proof}
\subsection{Maximal Characteristic}
\label{subsec: maxchar}
Recall the following theorem on $\PAINV_n(\y)$ for constant $\y$.
\begin{theorem}[Theorem~3.1 in \cite{icermpaper2023inv}]
\label{theorem: constant}
    Let $\y = (a^n) \in\N^n$ and $\x = (x_1, x_2, \ldots, x_n) \in [m]^n=[\sum_{i=1}^{n}y_i]^n$.
    We have $\x \in \PAINV_{n}(\y)$ if and only if
    \begin{equation}
    \label{eqn: constant}
         x_{(i)} \in \{ 1+(k-1)a:k \in [i] \} \quad \forall i \in [n].
    \end{equation}
\end{theorem}
We now perturb $\y$ and study length vectors of the form $(b,a^{n-1})$, where $a,b \in \N$ and $n\geq 2$. 
We will refer to such $\y$ as ``almost constant" (note that constant $\y$ also fall into this classification).

As seen in the statement of Theorem \ref{mainthm: almostconstant}, it turns out that even when $a\neq b$, under certain conditions, this perturbation of $\y$ does not affect $\PAINV_n(\y)$. 

First, we prove two intermediate steps for the proof of the closed forms of $\PAINV_n(\y)$ as described in Theorem \ref{mainthm: almostconstant}.
For brevity, for parking spots $S,E \in \N$, let the interval $[S,E]\coloneqq \{ S,S+1,\dots,E-1,E \}$ denote the set of parking spots numbered between $S$ and $E$, inclusive; define its length to be $|[S,E]|\coloneqq |\{ S,S+1,\dots,E-1,E \}|=E-S+1$.
% \begin{definition}
%     Let $S,E \in \N$. 
%     Then the interval
%     \[ [S,E]\coloneqq \{ S,S+1,\dots,E-1,E \}, \]
%     denotes the set of parking spots numbered between $S$ and $E$, inclusive. 
%     The length of this interval will be defined as 
%     \[ |[S,E]|\coloneqq |\{ S,S+1,\dots,E-1,E \}|=E-S+1. \]
% \end{definition} 
\begin{lemma}
\label{lemma: AC1moda}
    Let $\y = (b,a^{n-1}) \in\mathbb{N}^n$, where $n\geq 2$, and $\x = (x_1, x_2, \dots, x_n) \in [m]^n$.
    If $\x \in \PAINV_n(\y)$, then 
    \[ x_i\equiv 1 \pmod{a} \quad \forall i \in [n]. \]
\end{lemma}
\begin{proof}
    We will prove the contrapositive. 
    Assume that $x_j \not \equiv 1\pmod{a}$ for some $j \in [n]$.
    
    By Euclidean division, we may write $x_j=aq+r$, where $q \in \N_0\coloneqq \N \cup \{ 0 \}$, and $r\neq 1$ is a remainder (i.e. $r\in [a-1]_0 \setminus \{ 1 \}$). 
    Let $\x'=(x_1',x_2',\dots,x_n')$ be any permutation of $\x$ such that $x_1'=x_{j}$. 
    Under the parking experiment for $\x'$, the first car will occupy the $b$ spots $[aq+r,aq+r+b-1]$, which leaves $[1,aq+r-1]$ empty. 
    As $r\neq 1$, $|[1,aq+r-1]|$ is not an integral multiple of $a$.
    In particular, none of the remaining cars, which all have length $a$, can precisely fill $[1,aq+r-1]$. Thus, we have $\x' \notin \PA_n(\y)$ and $\x \notin \PAINV_n(\y)$. 
\end{proof}
\begin{lemma}
\label{lemma: orderstatbound}
    Let $\y=(b,a^{n-1}) \in \N^n$, where $n\geq 2$, and $\x = (x_1,x_2,\dots,x_n) \in [m]^n$.
    If $\x \in \PAINV_n(\y)$, then 
    \[ x_{(i)}\leq 1+(i-1)a \quad \forall i \in [n]. \]
\end{lemma}
\begin{proof}
    Again, we argue using the contrapositive.
    Assume that $x_{(j)}>1+(j-1)a$ for some $j \in [n]$.
    Then $x_{(j)}=1+qa$ for some $q \in [n-1]$ by Lemma \ref{lemma: AC1moda}.
    Let $\x'=(x_1',x_2',\dots,x_n')$ be any permutation of $\x$ such that $x_1'=x_{(j)}$.
    Under the parking experiment for $\x'$, since $x_1'=x_{(j)}\leq x_{(j+1)}\leq \cdots \leq x_{(n)}$, there are at least $n-j+1$ cars, including the first car with length $b$, preferring to park in $[1+qa,m]$. 
    The total length of these cars is then at least $b+(n-j)a$, so to ensure that they all can park, $|[1+qa,m]|\geq b+(n-j)a$ must hold. 
    However,
    \[ |[1+qa,m]|=|[1+qa,b+(n-1)a]|=b+(n-q-1)a<b+(n-j)a. \]
    Hence, $\x' \notin \PA_n(\y)$, so $\x \notin \PAINV_n(\y)$, as desired.
\end{proof}
We are now ready to prove the closed form of $\PAINV_n(\y)$.
\begin{theorem}
\label{theorem: ACmultiple}
    Let $\y = (b,a^{n-1}) \in\mathbb{N}^n$, where $n\geq 2$, and $\x = (x_1, x_2, \ldots, x_n) \in [m]^n$.
    If %$b$ is an integral multiple of $a$
    $a \mid b$, then $\x \in \PAINV_n(\y)$ if and only if 
    \begin{equation}
    \label{eqn: ACmultiple}
         x_{(i)} \in \{ 1+(k-1)a:k\in [i] \} \quad \forall i \in [n].
    \end{equation} 
\end{theorem}
\begin{proof}
    Let $\x \in \PAINV_n(\y)$. 
    Combining Lemmas \ref{lemma: AC1moda} and \ref{lemma: orderstatbound}, it follows that $\x$ satisfies (\ref{eqn: ACmultiple}).
    % We first prove via the contrapositive that if $\x \in \PAINV_n(\y)$, then (\ref{eqn: ACmultiple}) is satisfied. Write $b=ha$, where $h \in \N$. 
    % Assume that $\x$ does not satisfy (\ref{eqn: ACmultiple}).
    % That is, there exists $j \in [n]$ for which
    % \[ x_{(j)}\not\equiv 1 \pmod{a} \quad \text{or} \quad x_{(j)}>1+(j-1)a. \]
    
    % If $x_{(j)}\not\equiv 1 \pmod{a}$, then $\x \notin \PAINV_n(\y)$ by Lemma \ref{lemma: AC1moda}.  
    % Thus, we may assume that any entry of $\x$ is congruent to $1 \pmod{a}$.

    % If $x_{(j)}>1+(j-1)a$, then we can write $x_{(j)}=1+j'a$, where $j'\geq j$. 
    % Let $\x'$ be any permutation of $\x$ such that $x'_1=x_{(j)}$. 
    % We claim that $\x'$ causes parking to fail. 
    % Indeed, under its parking experiment, as $x_1'=x_{(j)}\leq x_{(j+1)}\leq \cdots\leq x_{(n)}$, there are at least $n-j+1$ cars, including the first car with length $b=ha$, preferring to park in $[1+j'a,m]$. 
    % This requires 
    % \[ |[1+j'a,m]|=|[1+j'a,(n-1)a+b]|=|[1+j'a,(n+h-1)a]|\geq ha+(n-j)a=(n+h-j)a, \]
    % but $|[1+j'a,m]|=(n+h-j'-1)a<(n+h-j)a$. 
    % Hence, $\x' \notin \PA_n(\y)$, as claimed, so $\x \notin \PAINV_n(\y)$.

    % This proves that if $\x \in \PAINV_n(\y)$, then (\ref{eqn: ACmultiple}) is satisfied.
    We now prove the converse. 
    Assume that (\ref{eqn: ACmultiple}) holds. 
    In other words, for all $i \in [n]$, $x_{(i)}\leq 1+(i-1)a$ and $x_i=1+t_ia$ for some $t_i\in [n-1]_0$. 
    We will show inductively that parking succeeds under $\x$.

    Under the parking experiment for $\x$, the first car, with length $b=ha$ and preference $x_1$, will park in spots $[x_1,x_1+ha-1]$, where $x_1=1+t_1a$ for some $t_1\in [n-1]_0$; that is, $[1+t_1a,(t_1+h)a]$ is occupied by car 1.
    Next, the second car, with length $a$ and preference $x_2$, will attempt to park in $[x_2,x_2+a-1]$, where $x_2=1+t_2a$ for some $t_2\in [n-1]_0$. 
    Note that (\ref{eqn: ACmultiple}) stipulates that at most one entry of $\x$ can be $1+(n-1)a$, so if $t_1=n-1$, then $t_2<t_1$, so that car 2 successfully parks in $[x_2,x_2+a-1]=[1+t_2a,(t_2+1)a]$.
    Now, assume $t_1<n-1$. 
    If $t_1+h\leq n-1$, then we have
    \[ t_2 \in [t_1-1]_0 \cup \{ t_1+h,t_1+h+1,\dots,n-1 \} \quad \text{or} \quad t_2 \in \{ t_1,t_1+1,\dots,t_1+h-1 \}. \] 
    In the former, car 2 successfully parks in $[x_2,x_2+a-1]=[1+t_2a,(t_2+1)a]$.
    In the latter, car 2 fails to park in $[x_2,x_2+a-1]=[1+t_2a,(t_2+1)a]$ since car 1 occupies $[1+t_1a,(t_1+h)a] \supseteq [1+t_2a,(t_2+1)a]$.
    Thus, car 2 must park in $[1+(t_1+h)a,(t_1+h+1)a]$.
    If $t_1+h>n-1$, then we have
    \[ t_2 \in [t_1-1]_0 \quad \text{or} \quad t_2 \in \{ t_1,t_1+1,\dots,n-1 \}. \]
    In the former, car 2 successfully parks in $[x_2,x_2+a-1]=[1+t_2a,(t_2+1)a]$.
    In the latter, car 2 again fails to park in $[x_2,x_2+a-1]$ for the same reason, so it must park in $[1+(t_1+h)a,(t_1+h+1)a]$.
    This proves that car $2$ will occupy $[S_2,S_2+a-1]$, where $S_2\equiv 1 \pmod{a}$, which will serve as our base case.
    
    Assume now that the first $k \in [n-1]$ cars are parked, and each car $j$, where $j \in [k] \setminus \{ 1 \}$, occupies $[S_j,S_j+a-1]$ for some $S_j\equiv 1 \pmod{a}$.
    %\textcolor{red}{[What happens in the case where $x_1=x_2=1$, the second car doesn't actually part in these spots... Does the proof break?]}
    Moreover, for $S_1=1+t_1a$, car 1 occupies $[S_1,S_1+ha-1]$.
    Thus, all the terms of the sequence $(S_1,S_2,\dots,S_k) \in \N^k$ are distinct, so that its nondecreasing rearrangement $(S_{(1)},S_{(2)},\dots,S_{(k)})$ is strictly increasing.
    Now, for any $j \in [k]$, define 
    \[ \ell(j)\coloneqq \begin{cases}
        ha & \text{if }S_{(j)}=S_1 \\
        a & \text{otherwise}.
    \end{cases} \]
    Then at this stage of the parking experiment, the set of unoccupied parking spots is
    \begin{equation}
    \label{eqn: unoccupied}
        \mathcal{U}\coloneqq [m] \setminus \bigcup_{j=1}^{k}[S_{(j)},S_{(j)}+\ell(j)-1]=\bigcup_{j=0}^{k}\mathcal{U}_j, \quad \text{where} \quad \mathcal{U}_j\coloneqq \begin{cases}
            [1,S_{(1)}-1] & \text{if }j=0 \\
            [S_{(j)}+\ell(j),S_{(j+1)}-1] & \text{if }j \in [k-1] \\
            [S_{(k)}+\ell(k),m] & \text{if }j=k.
        \end{cases}
    \end{equation}
    We note that the $\mathcal{U}_j$ are pairwise disjoint and possibly empty, and there should be a total of $(n+h-1)a-(h+k-1)a=(n-k)a$ unoccupied spots, so that $|\mathcal{U}|\coloneqq \sum_{j=0}^{k}|\mathcal{U}_j|=(n-k)a$.
    In particular, as the inductive hypothesis gives $S_{(j)}\equiv 1 \pmod{a}$, (\ref{eqn: unoccupied}) implies that any $\mathcal{U}_j$ has a left endpoint congruent to $1 \pmod{a}$, and $|\mathcal{U}_j|$ is an integral multiple of $a$.
    
    We now show that the $(k+1)$th car can park. 
    If $x_{k+1}=1$, then (\ref{eqn: unoccupied}) implies that car $k+1$ can find some interval of $a$ unoccupied spots to park in. 
    If $x_{k+1}=1+t_{k+1}a$ for some $t_{k+1} \in [n-1]$, then we have the following three distinct possibilities.
    \begin{case}
    \label{case: ACmultiple1}
        $t_1=n-1$ and $x_{k+1}=1+t_{k+1}a<1+t_1a$.
    \end{case}
    \begin{proof}
        Here, the first car occupies $[1+(n-1)a,(n+h-1)a]$. 
        If spot $1+t_{k+1}a$ is empty, then car $k+1$ will park in $\mathcal{P}\coloneqq [1+t_{k+1}a,(t_{k+1}+1)a]$ by (\ref{eqn: unoccupied}).
        Otherwise, we claim that car $k+1$ must be able to find and park in an interval of $a$ unoccupied spots in $\mathcal{A}\coloneqq [1+(t_{k+1}+1)a,(n-1)a]$. 
        \begin{figure}
            \centering
            \includegraphics[width=13.2cm]{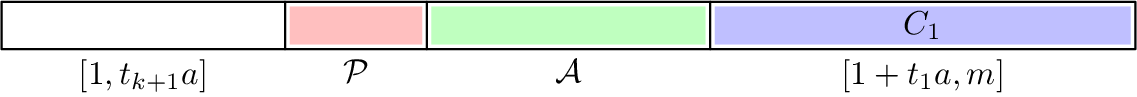}
            \caption{
            Illustration of Case \ref{case: ACmultiple1}.
            Car 1 takes the last $b$ spots in blue. 
            Next, car $k+1$ either parks according to its preference in the red, or it must be able to find an alternative spot in the green, which is located between the red and blue regions.
            }
        \end{figure}
        
        Assume the contrary, so that $\mathcal{A}$ is completely occupied. 
        Let 
        \begin{equation}
        \label{eqn: Smax}
            S_\text{max}=\max\{ s \in [m]:s \text{ is unoccupied when car $k+1$ attempts to park} \}.
        \end{equation}
        Since $S_\text{max}$ is the right endpoint of an interval of unoccupied spots, by (\ref{eqn: unoccupied}), $S_\text{max}=ta$ for some $t\in [t_{k+1}]$. 
        Then $[1+S_{\text{max}},(n+h-1)a]$, where $|[1+S_{\text{max}},(n+h-1)a]|=(n+h-t-1)a$, is occupied by $n-t$ cars, one of which is car 1 (with length $ha$) and $n-t-1$ of which are cars with length $a$. 
        Moreover, by assumption, $[1+S_{\text{max}},(n+h-1)a] \supseteq \mathcal{P}$.
        Thus, including car $k+1$, at least $n-t+1$ cars have preference at least $1+S_{\text{max}}$ since $S_{\text{max}}$ is unoccupied.
        In other words, at most $t-1$ cars have preference strictly less than $1+S_{\text{max}}$, forcing $x_{(t)}\geq 1+S_{\text{max}}=1+ta$.
        This contradicts $x_{(t)}\leq 1+(t-1)a$. 
    \end{proof}
    \begin{case}
    \label{case: ACmultiple2}
        $t_1\neq n-1$ and $x_{k+1}=1+t_{k+1}a<1+t_1a$.
    \end{case}
    \begin{proof}
        Now, the first car no longer occupies spot $m$.
        Note again by (\ref{eqn: unoccupied}) that if spot $1+t_{k+1}a$ is empty, then car $k+1$ will park in $\mathcal{P}\coloneqq [1+t_{k+1}a,(t_{k+1}+1)a]$. 
        Otherwise, car $k+1$ searches for an interval of $a$ unoccupied spots in $\mathcal{A}_1\coloneqq [1+(t_{k+1}+1)a,t_1a]$. 
        If it finds one, then it will occupy it. 
        If it does not, then we claim that it must be able to find and park in such an interval in $\mathcal{A}_2\coloneqq [1+(h+t_1)a,(n+h-1)a]$.
        \begin{figure}
            \centering
            \includegraphics[width=13.2cm]{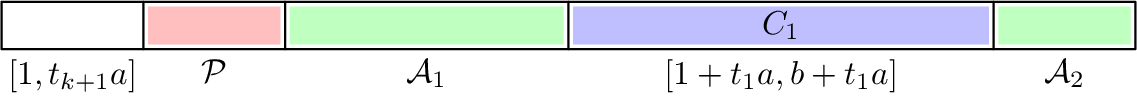}
            \caption{
            Illustration of Case \ref{case: ACmultiple2}.
            Car 1 takes the $b$ spots in blue. 
            Next, car $k+1$ either parks according to its preference in the red, or it must be able to find an alternative spot in one of the green regions, which surround car 1.
            }
        \end{figure}
        
        Assume the contrary, so that $\mathcal{A}_1$ and $\mathcal{A}_2$ are completely occupied. 
        Let $S_{\text{max}}$ be as in (\ref{eqn: Smax}), so that again, $S_{\text{max}}=ta$ for some $t \in [t_{k+1}]$.
        Repeating the argument for Case \ref{case: ACmultiple1} allows us to arrive at the same contradiction. 
        %For the same reason, $S_{\text{max}}=ta$ for some $t\in [t_{k+1}]$. 
        
        %Then $[1+S_{\text{max}},(n+h-1)a]$, where $|[1+S_{\text{max}},(n+h-1)a]|=(n+h-t-1)a$, is occupied by $n-t$ cars, one of which is car 1 (with length $ha$) and $n-t-1$ of which are cars with length $a$. 
        %Thus, including car $k+1$, at least $n-t+1$ cars have preference at least $S_{\text{max}}+1$.
        %In other words, at most $t-1$ cars have preference less than $S_\text{max}$, which forces $x_{(t)}\geq S_\text{max}+1=1+ta$, contradicting $x_{(t)}\leq 1+(t-1)a$. 
    \end{proof}
    \begin{case}
    \label{case: ACmultiple3}
        $x_{k+1}=1+t_{k+1}a\geq 1+t_1a$.
    \end{case}
    \begin{proof}
        Note first that $t_1<n-1$ due to (\ref{eqn: ACmultiple}). Here, we claim that car $k+1$ will find and park in an interval of $a$ unoccupied spots in $\mathcal{I}\coloneqq [1+(h+t_1)a,(n+h-1)a]$. 
        \begin{figure}
            \centering
            \includegraphics[width=13.2cm]{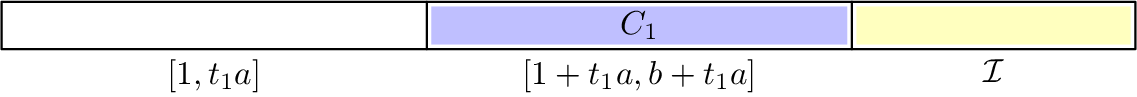}
            \caption{
            Illustration of Case \ref{case: ACmultiple3}.
            Car 1 takes the $b$ spots in blue. 
            Next, car $k+1$ will find parking in the yellow, which is located to the right of car 1.
            }
        \end{figure}
        
        Assume now that car $k+1$ cannot find such an interval, and let $S_{\text{max}}$ be as in (\ref{eqn: Smax}). 
        
        Note that we may take $h+t_1<n-1$.
        Otherwise, $h+t_1\geq n-1$, and since car 1 occupies $[1+t_1a,(h+t_1)a]$, any car with preference at least $1+t_1a$ is forced to park in an interval of $a$ unoccupied spots in $\mathcal{I}$.
        Then our assumption implies that $\mathcal{I}$ is completely occupied when car $k+1$ attempts to park, so that $S_{\text{max}}=ta$ for some $t \in [t_{k+1}]$.
        Repeating the argument for Case \ref{case: ACmultiple1} then completes the proof.
        For similar reasons, we may assume $S_{\text{max}}\geq 1+(h+t_1)a$.
        
        Given these additional assumptions, we have
        \[ S_{\text{max}}=ta \quad \text{for some} \quad t \in \{ h+t_1+1,h+t_1+2,\dots,t_{k+1} \}. \]
        Thus, $[1+S_{\text{max}},(n+h-1)a]$, where $|[1+S_\text{max},(n+h-1)a]|=(n+h-t-1)a$, is occupied by $n+h-t-1$ cars, all of which must have length $a$ since $S_{\text{max}}\geq 1+(h+t_1)a$. 
        Thus, including car $k+1$, at least $n+h-t$ cars have preference at least $1+S_{\text{max}}$ since $S_{\text{max}}$ is unoccupied, so at most $t-h$ cars have preference strictly less than $1+S_\text{max}$.
        Thus, $x_{(t-h+1)}\geq 1+S_\text{max}=1+ta$. 
        Since $h$ is a positive integer, this contradicts $x_{(t-h+1)}\leq 1+(t-h)a$.
    \end{proof}
    This covers all cases and completes the induction. 
    Hence, $\x \in \PA_n(\y)$. 
    Note that (\ref{eqn: ACmultiple}) does not depend on how the entries of $\x$ are permuted, so parking succeeds under any permutation of $\x$. 
    Therefore, if $\x$ satisfies (\ref{eqn: ACmultiple}), then $\x \in \PAINV_n(\y)$, concluding the proof.
\end{proof}
\begin{theorem}
\label{theorem: AClarge}
    Let $\y = (b,a^{n-1}) \in\mathbb{N}^n$, where $n\geq 2$, and $\x = (x_1, x_2, \ldots, x_n) \in [m]^n$.
    If %$b$ is not an integral multiple of $a$ with 
    $b>(n-1)a$, then $\x \in \PAINV_n(\y)$ if and only if 
    \begin{equation}
    \label{eqn: AClarge}
         x_{(i)} \in \{ 1+(k-1)a:k\in [i] \} \quad \forall i \in [n].
    \end{equation}
\end{theorem}
\begin{proof}
    Let $\x \in \PAINV_n(\y)$.
    Lemmas \ref{lemma: AC1moda} and \ref{lemma: orderstatbound} again show that $\x$ satisfies (\ref{eqn: AClarge}).
    % First, we prove that if $\x \in \PAINV_n(\y)$, then (\ref{eqn: AClarge}) is satisfied. 
    % Again, we consider the contrapositive and assume $\x$ does not satisfy (\ref{eqn: AClarge}). 
    % That is, there exists $j \in [n]$ for which
    % \[ x_{(j)}\not\equiv 1 \pmod{a} \quad \text{or} \quad x_{(j)}>1+(j-1)a. \]
    % We may assume any entry of $\x$ is congruent to $1 \pmod{a}$ by Lemma \ref{lemma: AC1moda}.
    
    % If $x_{(j)}>1+(j-1)a$, then we have $x_{(j)}=1+j'a$, where $j'\geq j$. 
    % We claim that any permutation $\x'$ of $\x$ such that $x'_1=x_{(1)}=1$ and $x'_2=x_{(j)}$ causes parking to fail. 
    % Indeed, under the parking experiment for $\x'$, the first car occupies $[1,b]$, and the second car occupies $[x_{(j)},x_{(j)}+a-1]$, leaving $[b+1,x_{(j)}-1]$ completely unoccupied with $|[b+1,x_{(j)}-1]|=x_{(j)}-b-1=j'a-b$. 
    % Since $b$ is not an integral multiple of $a$,
    % \[ 0<ja-b\leq na-b<na-(n-1)a=a. \] 
    % Consequently, we have
    % \[ (j'-j)a<j'a-b\leq na+(j'-j)a-b<(j'-j)a+a=(j'-j+1)a, \]
    % so $j'a-b$ is strictly bounded between consecutive multiples of $a$, meaning that none of the remaining cars, which all have length $a$, can completely fill the gap left by the initial two cars, as claimed.
    
    % This shows that any $\x$ not satisfying the conditions is not invariant. 
    We now prove the converse.
    Assume that (\ref{eqn: AClarge}) holds.
    In other words, for all $i \in [n]$, $x_{(i)}\leq 1+(i-1)a$ and $x_i=1+t_ia$ for some $t_i\in [n-1]_0$. 
    We will show inductively that parking succeeds under $\x$.

    Under the parking experiment for $\x$, the first car, with length $b$ and preference $x_1$, will park in $[x_1,x_1+b-1]$, where $x_1=1+t_1a$ for some $t_1\in [n-1]_0$. 
    Thus, $[1+t_1a,b+t_1a]$ is now occupied.
    Next, the second car, with length $a$ and preference $x_2$, will attempt to park in $[x_2,x_2+a-1]$, where $x_2=1+t_2a$ for some $t_2\in [n-1]_0$. 
    Due to (\ref{eqn: AClarge}), if $t_1=n-1$, then $t_2<t_1$, so that car $2$ successfully parks in $[x_2,x_2+a-1]=[1+t_2a,(t_2+1)a]$. 
    If $t_1<n-1$, then we have
    \[ t_2 \in [t_1-1]_0 \quad \text{or} \quad t_2 \in \{ t_1,t_1+1,\dots,n-1 \}. \]
    %consider the possibilities $t_2 \in \{ 0,1,\dots,t_1-1 \}$ or $t_2 \in \{ t_1,t_1+1,\dots,n-1 \}$.
    In the former, car 2 successfully parks in $[x_2,x_2+a-1]=[1+t_2a,(t_2+1)a]$.
    In the latter, car 2 fails to park in $[x_2,x_2+a-1]=[1+t_2a,(t_2+1)a]$ since car 1 occupies $[1+t_1a,b+t_1a]$, where $b+t_1a>(n+t_1-1)a\geq (n-1)a$, so $b+t_1a\geq 1+(n-1)a$ and $b+t_1a\geq 1+t_2a$.
    Thus, car 2 must park in $[1+b+t_1a,b+(t_1+1)a]$. 
    This proves that car 2 will occupy $[S_2,S_2+a-1]$, where 
    \[ S_2\equiv \begin{cases}
        1 \pmod{a} & \text{if }S_2<1+t_1a \\
        1+b \pmod{a} & \text{if }S_2>b+t_1a,
    \end{cases} \]
    which will serve as our base case.
    
    Assume now that the first $k \in [n-1]$ cars are parked, and each car $j$, where $j \in [k] \setminus \{ 1 \}$, occupies $[S_j,S_j+a-1]$ for some
    \[ S_j\equiv \begin{cases}
        1 \pmod{a} & \text{if }S_j<1+t_1a \\
        1+b \pmod{a} & \text{if }S_j>b+t_1a.
    \end{cases} \]
    In other words, $S_j\equiv 1 \pmod{a}$ if it is located to the left of $[1+t_1a,b+t_1a]$ (where car 1 is parked), and $S_j \equiv 1+b \pmod{a}$ otherwise.
    Moreover, for $S_1=1+t_1a$, car 1 occupies $[S_1,S_1+b-1]$.
    Now, for any $j \in [k]$, define
    \begin{equation} 
    \label{eqn: lengthnonmultiple}
        \ell(j)\coloneqq \begin{cases}
            b & \text{if $S_{(j)}=S_1$} \\
            a & \text{otherwise}.
        \end{cases} 
    \end{equation}
    At this stage of the parking experiment, the set of unoccupied spots is again $\bigcup_{j=0}^{k}\mathcal{U}_j$, where the $\mathcal{U}_j$ are defined as in (\ref{eqn: unoccupied}) with $\ell(j)$ as in (\ref{eqn: lengthnonmultiple}).
    Furthermore, from (\ref{eqn: unoccupied}), any $\mathcal{U}_j$ has a left endpoint congruent to either $1\pmod{a}$ or $1+b\pmod{a}$ (when $\mathcal{U}_j$ is located to the left and right, respectively, of where car 1 is parked), and $|\mathcal{U}_j|$ is an integral multiple of $a$.
    %This means that the first car occupies $[1+t_1a,t_1a+b]$, and car $i$, where $1<i\leq n$, occupies spots of the form $[1+t_ia,(t_i+1)a]$ (parking to the left of the first car) or $[1+t_ia+b,(t_i+1)a+b]$ (parking to the right of the first car), for $0\leq t_i\leq n-1$. Thus, every block of unoccupied spots has a length that is a multiple of $a$, and the first available spot of the block is either $1$ or $1+b$ more than a multiple of $a$.

    We now show that the $(k+1)$th car can park. If $x_{k+1}=1$, then (\ref{eqn: unoccupied}) implies that car $k+1$ can find some interval of $a$ unoccupied spots to park in. 
    If $x_{k+1}=1+t_{k+1}a$ for $t_{k+1}\in [n-1]$, then we have the following three distinct possibilities.
    \begin{case}
    \label{case: AClarge1}
        $t_1=n-1$ and $x_{k+1}=1+t_{k+1}a<1+t_1a$.
    \end{case}
    \begin{proof}
        Here, the first car occupies $[1+(n-1)a,b+(n-1)a]$.
        %This implies that the first car takes the last $b$ spots (i.e. spots $[1+(n-1)a,b+(n-1)a]$). 
        If spot $1+t_{k+1}a$ is empty, then car $k+1$ will park in $[1+t_{k+1}a,(t_{k+1}+1)a]$ by (\ref{eqn: unoccupied}).
        Otherwise, we claim that car $k+1$ must be able to find and park in an interval of $a$ unoccupied spots in $[1+(t_{k+1}+1)a,(n-1)a]$. 
        Assume not, so that $[1+(t_{k+1}+1)a,(n-1)a]$ is completely occupied. 
        Define $S_{\text{max}}$ as in (\ref{eqn: Smax}).
        Since $S_{\text{max}}$ is the right endpoint of an interval of unoccupied spots located to the left of where car 1 is parked, by (\ref{eqn: unoccupied}), $S_{\text{max}}=ta$ for some $t\in [t_{k+1}]$. 
        Then $[1+S_{\text{max}},b+(n-1)a]$, where $|[1+S_{\text{max}},b+(n-1)a]|=(n-t-1)a+b$, are occupied by $n-t$ cars, one of which is car 1 (with length $b$) and $n-t-1$ of which are cars with length $a$. 
        Thus, including car $k+1$, at least $n-t+1$ cars have preference at least $1+S_{\text{max}}$.
        Hence, at most $t-1$ cars have preference strictly less than $1+S_{\text{max}}$, forcing $x_{(t)}\geq 1+S_{\text{max}}=1+ta$, which contradicts $x_{(t)}\leq 1+(t-1)a$.
    \end{proof}
    \begin{case}
        $t_1\neq n-1$ and $x_{k+1}=1+t_{k+1}a<1+t_1a$.
    \end{case}
    \begin{proof}
        Note again that (\ref{eqn: unoccupied}) implies that if spot $1+t_{k+1}a$ is empty, then car $k+1$ will park in $[1+t_{k+1}a,(t_{k+1}+1)a]$.
        Otherwise, car $k+1$ searches for an interval of $a$ unoccupied spots in $[1+(t_{k+1}+1)a,t_1a]$. 
        If it finds one, then it will occupy it. 
        If it does not, then we claim that it must be able to find and park in such an interval in $[1+b+t_1a,b+(n-1)a]$. 
        Assume not, so that both $[1+(t_{k+1}+1)a,t_1a]$ and $[1+b+t_1a,b+(n-1)a]$ are completely occupied.
        Let $S_{\text{max}}$ be as in (\ref{eqn: Smax}), so that by (\ref{eqn: unoccupied}) again, $S_{\text{max}}=ta$ for some $t\in [t_{k+1}]$. 
        Repeating the argument for Case \ref{case: AClarge1} yields the same contradiction.
        %This means that spots $[s+1,(n-1)a+b]$, which is a total of $(n-t-1)a+b$ spots, are all occupied by $n-t$ cars, one of which is the first (with length $b$) and the rest of which are the others (all with length $a$). Thus, including car $k+1$, there are at least $n-t+1$ cars that have preference at least $s+1$; in particular, there are at most $t-1$ cars that have preference less than $s$, which forces $x_{(t)}\geq s+1=1+ta$, contradicting $x_{(t)}\leq 1+(t-1)a$. 
    \end{proof}
    \begin{case}
        $x_{k+1}=1+t_{k+1}a\geq 1+t_1a$.
    \end{case}
    \begin{proof}
        Here, we claim that car $k+1$ will be able to find and park in an interval of $a$ unoccupied spots in $[1+b+t_1a,b+(n-1)a]$.
        Assume not, and let $S_{\text{max}}$ be as in (\ref{eqn: Smax}). 
        Note that $b+t_1a\geq 1+(n-1)a$, so any car with preference at least $1+t_1a$ is forced to park in an interval of $a$ unoccupied spots in $[1+b+t_1a,b+(n-1)a]$.
        Then our assumption implies that $[1+b+t_1a,b+(n-1)a]$ is completely occupied when car $k+1$ attempts to park, so that $S_{\text{max}}=ta$ for some $t \in [t_{k+1}]$. Repeating the argument for Case \ref{case: AClarge1} yields another contradiction.
        %as it is the endpoint of a block of unoccupied spots to the right of where the first car is parked, by the inductive hypothesis, $s=ta+b$ for some $t_1<t\leq t_{k+1}$. This means that spots $[s+1,(n-1)a+b]$, which is a total of $(n-t-1)a$ spots, are occupied by $n-t-1$ cars, all of which have length $a$. Note that
        %\[ \left(t+\floor{\frac{b}{a}}\right)a<s=ta+b<\left(t+\floor{\frac{b}{a}}+1\right)a. \]
        %Thus, including car $k+1$, there are at least $n-t$ cars that have preference at least $1+\left(t+\floor{\frac{b}{a}}+1\right)a$ (i.e. greater than $s$). But by the conditions, no car can have preference greater than $1+(n-1)a$, and 
        %\[ t+1>0 \implies 1+\left(t+\floor{\frac{b}{a}}+1\right)a>1+\floor{\frac{b}{a}}a>1+(n-1)a, \]
        %so we have a contradiction. 
    \end{proof}
    This covers all cases and completes the induction, so $\x \in \PA_n(\y)$.
    Again, (\ref{eqn: AClarge}) is preserved under any permutation of $\x$.
    Therefore, if $\x$ satisfies (\ref{eqn: AClarge}), then $\x \in \PAINV_n(\y)$, concluding the proof.
\end{proof}
From Theorems \ref{theorem: ACmultiple} and \ref{theorem: AClarge}, we see that the conditions (\ref{eqn: constant}), (\ref{eqn: ACmultiple}), and (\ref{eqn: AClarge}) are identical, and hence the equality $\PAINV_n((a^n))=\PAINV_n((b,a^{n-1}))$ holds whenever $a \mid b$ or $b>(n-1)a$.

It remains to study the possibility of $b \nmid a$ and $b<(n-1)a$. 
For this case, we will see that the order statistic bounds provided by Lemma \ref{lemma: orderstatbound} are not strong enough to guarantee that $\x \in \PAINV_n(\y)$.
Thus, sharpening some of these bounds is needed to obtain such a result.
\begin{theorem}
\label{theorem: ACsmall}
    Let $\y = (b,a^{n-1}) \in\mathbb{N}^n$, where $n\geq 2$, and $\x = (x_1, x_2, \ldots, x_n) \in [m]^n$. 
    If $a \nmid b$ and $b<(n-1)a$, then $\x \in \PAINV_n(\y)$ if and only if 
    \begin{equation}
    \label{eqn: ACsmall}
        x_{(i)} \in 
        \begin{cases} 
            \{ 1+(k-1)a:k\in [i] \} & \forall i\in\left[\floor{\frac{b}{a}}+1\right] \\
            \left\{ 1+(k-1)a:k \in \left[ \floor{\frac{b}{a}}+1 \right] \right\} & \text{otherwise}.
        \end{cases}
    \end{equation}
\end{theorem}
\begin{proof}
    First, assume that $\x$ does not satisfy (\ref{eqn: ACsmall}). 
    That is, there exists $j \in [n]$ for which at least one of the following is true:
    \begin{align*} 
        x_{(j)}&\not\equiv 1 \pmod{a} \\
        x_{(j)}&>1+(j-1)a \text{ if }j\leq \floor{\frac{b}{a}}+1 \\
        x_{(j)}&>1+\floor{\frac{b}{a}} a \text{ if }j>\floor{\frac{b}{a}}+1. 
    \end{align*}
    Lemma \ref{lemma: AC1moda} allows us to assume that any entry of $\x$ is congruent to $1 \pmod{a}$.
    Similarly, Lemma \ref{lemma: orderstatbound} allows us to assume that $x_{(j)}\leq 1+(j-1)a$ if $j\leq \floor{\frac{b}{a}}+1$.

    % If $x_{(j)}>1+(j-1)a$, where $j\leq \floor{\frac{b}{a}}+1$, then we have $x_{(j)}=1+j'a$, where $j'\geq j$. 
    % We claim that the permutation $\x'$ of $\x$ such that $x'_1=x_{(1)}$ and $x'_2=x_{(j)}$ allows parking to fail. 
    % Indeed, car 1 occupies $[1,b]$, and car 2 fills spots $x_{(j)},\dots,x_{(j)}+a-1$, which leaves the $x_{(j)}-b-1=j'a-b$ spots $b+1,\dots,x_{(j)}-1$ empty. Note that
    % \[ 0<ja-b\leq \floor{\frac{b}{a}}a+a-b<a, \]
    % as by assumption, $b$ is not a multiple of $a$, so $a\floor{\frac{b}{a}}<b$ (we do not have equality, as otherwise, $\floor{\frac{b}{a}}=\frac{b}{a}$, so $b$ is a multiple of $a$). Consequently, we have
    % \[ (j'-j)a<j'a-b\leq \floor{\frac{b}{a}}a+(j'-j)a+a-b<a+(j'-j)a=(j'-j+1)a, \]
    % so $j'a-b$ is strictly bounded between consecutive multiples of $a$, meaning that none of the remaining cars, which all have length $a$, can completely fill the gap left by the initial two cars, as claimed.
    If $x_{(j)}>1+\floor{\frac{b}{a}}a$, where $j>\floor{\frac{b}{a}}+1$, then write $x_{(j)}=1+qa$, where $q\geq \floor{\frac{b}{a}}+1$. 
    Let $\x'=(x_1',x_2',\dots,x_n')$ be any permutation of $\x$ such that $x_1'=x_{(1)}$ and $x_2'=x_{(j)}$.
    Under the parking experiment for $\x'$, the first car will occupy the $b$ spots $[1,b]$, and the second car will occupy the $a$ spots $[x_{(j)},x_{(j)}+a-1]=[1+qa,(q+1)a]$.
    This leaves the $qa-b$ spots $[b+1,qa]$ empty. 
    Note that as $a \nmid b$, we have $\frac{b}{a}-1<\floor{\frac{b}{a}}<\frac{b}{a}$, so
    \[ qa-b\geq \left( \floor{\frac{b}{a}}+1 \right)a-b>\frac{ab}{a}-b=0; \]
    thus, $[b+1,qa]\neq \emptyset$.
    %Note that
    %\[ 0<\floor{\frac{b}{a}}a+a-b<a, \]
    %as by assumption, $a \nmid b$, so $0<a\floor{\frac{b}{a}}$ 
    %(we do not have equality, as otherwise, $\frac{b}{a}-\floor{\frac{b}{a}}=1$, which is absurd). 
    Similarly, $\floor{\frac{b}{a}}<\frac{b}{a}<\floor{\frac{b}{a}}+1$, which yields
    \[ \left(q-\floor{\frac{b}{a}}-1\right)a<qa-b<\left(q-\floor{\frac{b}{a}}\right)a, \]
    so $qa-b$ is strictly bounded between consecutive multiples of $a$. 
    Hence, none of the remaining cars, which all have length $a$, can fill $[b+1,qa]$. 
    Thus, we have $\x' \in \PA_n(\y)$ and $\x \notin \PAINV_n(\y)$.

    Thus, if $\x \in \PAINV_n(\y)$, then $\x$ satisfies (\ref{eqn: ACsmall}). 
    
    We now prove the converse. 
    Assume that (\ref{eqn: ACsmall}) holds. 
    In other words, for all $i\in \left[\floor{\frac{b}{a}}+1\right]$, $x_{(i)}\leq 1+(i-1)a$, for all $\floor{\frac{b}{a}}+1<i\leq n$, $x_{(i)}\leq 1+\floor{\frac{b}{a}}a$, and $x_i=1+t_ia$ for some $t_i\in \left[\floor{\frac{b}{a}}\right]_0$. 
    We will show inductively that parking succeeds under $\x$.
    
    Under the parking experiment for $\x$, the first car, with length $b$ and preference $x_1$, will park in $[x_1,x_1+b-1]$, where $x_1=1+t_1a$ for some $t_1\in \left[\floor{\frac{b}{a}}\right]_0$. 
    Then, the second car, with length $a$ and preference $x_2$, will attempt to park in $[x_2,x_2+a-1]$, where $x_2=1+t_2a$ for some $t_2\in \left[ \floor{\frac{b}{a}} \right]_0$.
    We have 
    \[ t_2 \in [t_1-1]_0 \quad \text{or} \quad t_2 \in \left\{ t_1,t_1+1,\dots,\floor{\frac{b}{a}} \right\}. \]
    In the former, car 2 successfully parks in $[x_2,x_2+a-1]=[1+t_2a,(t_2+1)a]$.
    In the latter, car 2 fails to park in $[x_2,x_2+a-1]=[1+t_2a,(t_2+1)a]$ since car 1 occupies $[1+t_1a,b+t_1a]$, where $b+t_1a\geq b>\floor{\frac{b}{a}}a$, so $b+t_1a\geq 1+\floor{\frac{b}{a}}a$ and $b+t_1a\geq 1+t_2a$.
    Thus, car 2 must park in $[1+b+t_1a,b+(t_1+1)a]$.
    This proves that car 2 will occupy $[S_2,S_2+a-1]$, where
    \[ S_2 \equiv \begin{cases}
        1 \pmod{a} & \text{if }S_2<1+t_1a \\
        1+b \pmod{a} & \text{if }S_2>b+t_1a,
    \end{cases} \]
    which will serve as our base case.
    % If $b+t_1a<1+\floor{\frac{b}{a}}a$, then we have
    % \[ t_2 \in [t_1-1]_0 \cup \left\{ u,u+1,\dots,\floor{\frac{b}{a}} \right\} \quad \text{or} \quad t_2 \in \{ t_1,t_1+1,\dots,u-1 \}, \]
    % where $u=\min\{ t \in \N: 1+ta>b+t_1a \}$.
    % In the former, car 2 successfully parks in $[x_2,x_2+a-1]=[1+t_2a,(t_2+1)a]$.
    % In the latter, car 2 fails to park in $[x_2,x_2+a-1]=[1+t_2a,(t_2+1)a]$ since car $1$ occupies $[1+t_1a,b+t_1a]$, and $1+t_2a \in [1+t_1a,b+t_1a]$ by construction. 
    % Thus, car 2 must park in $[1+b+t_1a,b+(t_1+1)a]$.
    Assume now that the first $k \in [n-1]$ cars are parked, and each car $j$, where $j \in [k] \setminus \{ 1 \}$, occupies $[S_j,S_j+a-1]$ for some
    \[ S_j \equiv \begin{cases}
        1 \pmod{a} & \text{if }S_j<1+t_1a \\
        1+b \pmod{a} & \text{if }S_j>b+t_1a.
    \end{cases} \]
    Moreover, for $S_1=1+t_1a$, car 1 occupies $[S_1,S_1+b-1]$.
    At this stage of the parking experiment, the set of unoccupied spots is again $\bigcup_{j=0}^{k}\mathcal{U}_j$, where the $\mathcal{U}_j$ are defined as in (\ref{eqn: unoccupied}) with $\ell(j)$ as in (\ref{eqn: lengthnonmultiple}).
    Furthermore, from (\ref{eqn: unoccupied}), any $\mathcal{U}_j$ has a left endpoint congruent to either $1\pmod{a}$ or $1+b\pmod{a}$, and $|\mathcal{U}_j|$ is an integral multiple of $a$.
    % This means that the first car occupies $[1+t_1a,t_1a+b]$, and car $i$, where $1<i\leq n$, occupies spots of the form $[1+t_ia,(t_i+1)a]$ (parking to the left of the first car) or $[1+t_ia+b,(t_i+1)a+b]$ (parking to the right of the first car), for $0\leq t_i\leq n-1$. Thus, every block of unoccupied spots has a length that is a multiple of $a$, and the first available spot of the block is either $1$ or $1+b$ more than a multiple of $a$.

    We now show that the $(k+1)$th car can park. If $x_{k+1}=1$, then (\ref{eqn: unoccupied}) implies that car $k+1$ can find some interval of $a$ unoccupied spots to park in. 
    If $x_{k+1}=1+t_{k+1}a$ for $t_{k+1}\geq 1$, then we have the following two distinct possibilities.
    \begin{case}
        $x_{k+1}=1+t_{k+1}a<1+t_1a$. 
    \end{case}
    \begin{proof}
        Note first that car 1 cannot occupy $[1+(n-1)a,b+(n-1)a]$ since $x_1\leq 1+\floor{\frac{b}{a}}a<1+(n-1)a$.
        %If it did, then its preference is at most $1+\floor{\frac{b}{a}}a$, so it fills spots $[1+\floor{\frac{b}{a}}a,b+\floor{\frac{b}{a}}a]$, and $b+\floor{\frac{b}{a}}a=b+(n-1)a$, which implies $\floor{\frac{b}{a}}=n-1$, a contradiction.
        In particular, as $\floor{\frac{b}{a}}\leq n-2$, we have $|[x_1+b,b+(n-1)a]|=(n-1)a-\floor{\frac{b}{a}}a\geq a$; in other words, the first car will park at least $a$ spots away from the end of the parking lot.

        Now, note that if spot $1+t_{k+1}a$ is empty, then car $k+1$ will park in $[1+t_{k+1}a,(t_{k+1}+1)a]$ by (\ref{eqn: unoccupied}).
        Otherwise, it searches for an interval of $a$ unoccupied spots in $[1+(t_{k+1}+1)a,t_1a]$.
        %(the remaining spots to the left of where the first car is parked). 
        If it finds one, then it will occupy it.
        If it does not, then we claim it must be able to find and park in such an interval in $[1+t_1a+b,b+(n-1)a]$. 
        
        Assume not, so that all of these spots are occupied. 
        Let $S_{\text{max}}$ be as in (\ref{eqn: Smax}).
        Since it is the right endpoint of an interval of unoccupied spots located to the left of where car 1 is parked, by (\ref{eqn: unoccupied}), $S_{\text{max}}=ta$ for some $t\in [t_{k+1}]$. 
        Repeating the argument for Case \ref{case: AClarge1} yields the same contradiction.
        % Then $[1+S_{\text{max}},b+(n-1)a]$, where $[1+S_{\text{max}},b+(n-1)a]=b+(n-t-1)a$, are all occupied by $n-t$ cars, one of which is car 1 (with length $b$) and $n-t-1$ of which are cars with length $a$. 
        % Thus, including car $k+1$, at least $n-t+1$ cars have preference at least $1+S_{\text{max}}$.
        % Hence, at most $t-1$ cars have preference strictly less than $1+S_{\text{max}}$, forcing $x_{(t)}\geq 1+S_{\text{max}}=1+ta$, which contradicts $x_{(t)}\leq 1+(t-1)a$.
    \end{proof}
    \begin{case} 
        $x_{k+1}=1+t_{k+1}a\geq 1+t_1a$.
    \end{case}
    \begin{proof}
        Here, we claim that car $k+1$ will be able to find and park in an interval of $a$ unoccupied spots in $[1+t_{k+1}a,b+(n-1)a]$. 
        Assume not, and let $S_{\text{max}}$ be as in (\ref{eqn: Smax}). 
        Since it is the right endpoint of an interval of unoccupied spots located to the right of where car 1 is parked, by (\ref{eqn: unoccupied}), $S_{\text{max}}=ta+b$ for some $t_1<t< t_{k+1}$. 
        %Then $[1+S_{\text{max}},(n-1)a+b]$, where $|[1+S_{\text{max}},(n-1)a+b]|=(n-t-1)a$ spots, are all occupied by $n-t-1$ cars with length $a$.
        We have
        \[ \left(t+\floor{\frac{b}{a}}\right)a<S_{\text{max}}=ta+b<\left(t+\floor{\frac{b}{a}}+1\right)a. \]
        Thus, 
        %including car $k+1$, at least $n-t$ cars 
        as $S_{\text{max}}$ is unoccupied, there are cars that have preference at least $1+\left(t+\floor{\frac{b}{a}}+1\right)a>S_{\text{max}}$. 
        But by (\ref{eqn: ACsmall}), no car can have preference greater than $1+\floor{\frac{b}{a}}a$, and $1+\left(t+\floor{\frac{b}{a}}+1\right)a>1+\floor{\frac{b}{a}}a$, so we have a contradiction. 
    \end{proof}
    This covers all cases and completes the induction, so $\x \in \PA_n(\y)$. 
    Again, (\ref{eqn: ACsmall}) is preserved under any permutation of $\x$.
    Therefore, if $\x$ satisfies (\ref{eqn: ACsmall}), then $\x \in \PAINV_n(\y)$, concluding the proof.
\end{proof}
Theorems \ref{theorem: ACmultiple}, \ref{theorem: AClarge}, and \ref{theorem: ACsmall} illustrate that for almost constant $\y$, $\chi(\y)$ is maximal when $y_1\geq y_2$ and minimal otherwise.
%Thus, we have shown that almost constant $\y$, apart from those that satisfy $y_1<y_2$, are examples of length vectors with maximal characteristic.
We now prove a surprising partial converse: the almost constant condition on length vectors is necessary for their characteristic to be maximal.
Recall the following results.
% \begin{lemma}
% \label{lemma: twocars}
%     Let $\y \in \N^2$.
%     Then 
%     \[ \PAINV_2(\y)=\begin{cases}
%         \{ (1^2) \} & \text{if $y_1<y_2$} \\
%         \{ (1^2),(1,1+y_2),(1+y_2,1) \} & \text{if $y_1\geq y_2$}.
%     \end{cases} \]
% \end{lemma}
\begin{lemma}[Lemma 2.3 in ~\cite{icermpaper2023inv}]
\label{lemma: removal}
    Let $\y \in \N^n$ and $\x=(x_1,x_2,\dots,x_n) \in \PAINV_n(\y)$. 
    If $k=\argmax_{i \in [n]}x_i$, then $\x_{\widehat{k}} \in \PAINV_{n-1}(\y_{\vert_{n-1}})$. 
    In particular, if $\chi(\y)=\alpha$, then $\chi(\y_{\vert_{n-1}})\geq \alpha-1$.
\end{lemma}
\begin{lemma}[Theorem 5.1 in ~\cite{icermpaper2023inv}]
\label{lemma: twocarschar}
    Let $\y=(y_1,y_2) \in \N^2$.
    Then $\chi(\y)=1$ if and only if $y_1\geq y_2$.
\end{lemma}
\begin{lemma}[Theorem 5.4 in ~\cite{icermpaper2023inv}]
\label{lemma: charmaximpliesACbasecase}
    Let $\y=(y_1,y_2,y_3) \in \N^3$.
    If $\chi(\y)=2$, then $y_2=y_3$.
\end{lemma}
Now, we will present the partial converse, which implies one direction of Theorem \ref{mainthm: maxchar}.
Its proof will be via induction on $n$, so the above results will be instrumental in setting up the base cases and proving the inductive step.
\begin{theorem}
\label{theorem: charmaximpliesAC}
    Let $\y=(y_1,y_2,\dots,y_n)\in \mathbb{N}^n$, where $n\geq 2$.
    If $\chi(\y)=n-1$, then 
    \[ y_2=y_3=\cdots=y_n. \]
    %$y_2=y_3=\cdots=y_n$.
    %If $\x \in \PAINV_n(\y)$, and $\deg \x=n-1$, then $y_2=y_3=\cdots=y_n$.
\end{theorem}
\begin{proof}
    We will proceed via induction on $n$. 
    Note that $\chi(\y)=n-1$ implies that there exists $\mathbf{w} \in \N_{>1}^{n-1}$ such that $(1,\mathbf{w}) \in \PAINV_n(\y)$.
    %For $n=1$, there does not exist $(1,\mathbf{w}) \in \PAINV_1(\y)$, where all of the entries of $\mathbf{w}$ are greater than $1$, so the conclusion follows by vacuous truth. 
    For $n=2$, if $(1,w) \in \PAINV_2(\y)$, where $w>1$, then the conclusion follows vacuously by Lemma \ref{lemma: twocarschar}. 
    For $n=3$, the conclusion follows by Lemma \ref{lemma: charmaximpliesACbasecase}.
    
    Thus, we have verified the base cases, so assume that the statement is true up to some $n=k\geq 3$. 
    That is, for any $\y \in \mathbb{N}^k$ such that $\chi(\y)=k-1$, 
    %$(1,\mathbf{w}) \in \PAINV_k(\y)$ for some $\mathbf{w} \in \N_{>1}^{k-1}$, 
    we have $y_2=y_3=\cdots=y_k$. 
    We now show that this is true for $n=k+1$. 
    Let $\y'=(y_1',y_2',\dots,y_{k+1}') \in \mathbb{N}^{k+1}$ such that $\chi(\y')=k$. 
    Then there exists $\x=(1,\mathbf{w}') \in \PAINV_{k+1}(\y')$, where $\mathbf{w}' \in \N_{>1}^k$. 
    Assume without loss of generality that $\mathbf{w}'$ is nondecreasing. 
    As $w'_{k}=\max \x$, by Lemma \ref{lemma: removal}, we have $(1,\mathbf{w}'_{\vert{k-1}}) \in \PAINV_{k}(\y'_{\vert_k})$.
    Thus, $\chi(\y'_{\vert_k})\leq k-1$, and we have equality since $k-1$ is the maximal characteristic.
    Hence, by the inductive hypothesis, $\y'_{\vert_k}$ is of the form $(b,a^{k-1}) \in \N^k$, which shows that $y'_2=y'_3=\cdots=y'_k$. 
    Consequently, $\y'$ is of the form $(b,a^{k-1},c)$ for some $c \in \N$, and it remains to show that $c=a$. 
    
    To do this, we prove the contrapositive: if $c\neq a$, then $\x=(1,\mathbf{w}') \not\in \PAINV_{k+1}(\y')$. 
    We claim that parking fails for the permutation $\x'=(\mathbf{w}',1)$. 
    This can be rewritten as 
    \[ \x'=(w'_1,w'_2,\dots,w'_{k-1},w'_k,1)=(\mathbf{w}'_{\vert_{k-1}},w'_k,1). \] 
    Since $(1,\mathbf{w}'_{\vert{k-1}}) \in \PAINV_{k}(\y'_{\vert_k})$, we have $(\mathbf{w}'_{\vert_{k-1}},1) \in \PA_{k}(\y'_{\vert_k})$.
    Thus, consider the parking experiment for $(\mathbf{w}'_{\vert_{k-1}},1)$ with length vector $\y'_{\vert_k}$.
    As all cars park successfully and $\mathbf{w}' \in \N_{>1}^{k}$, car $k$, with preference $1$, must occupy $[1,a]$, while the remaining cars completely occupy $[1+a,m-c]$.
    % In other words, if the first $k-1$ cars have precisely the preferences $\mathbf{w}'_{\vert_{k-1}}$, then $[1,a]$ remains unoccupied. 
    This means that in the parking experiment for $\x'=(\mathbf{w}'_{\vert_{k-1}},w'_k,1)$ (with length vector $\y'$),
    % We now apply this fact to run the parking experiment for $(w'_1,\dots,w'_{k-1},w'_k,1)$. 
    the first $k-1$ cars will completely occupy $[1+a,m-c]$, so they leave $[1,a]$ and $[m-c+1,m]$ unoccupied. 
    By assumption, $w'_{k}=\max \x'>1$, so car $k$, with length $a$, must occupy a subset of $[m-c+1,m]$ (%we assume that $w'_k$ is small enough for the car to be able to park; 
    otherwise, parking fails, and the proof is complete). 
    
    We now consider the following two distinct cases.
    \begin{case}
         $\y'=(b,a^{k-1},c)$ with $a<c$.
    \end{case}
    \begin{proof}
        Here, car $k$ will take $a$ spots in $[m-c+1,m]$; thus, $c-a$ spots in this interval are still left unoccupied.
        Then, car $k+1$, with preference $1$ and length $c>a$, cannot park in $[1,a]$, and $c-a<c$, so it cannot park in $[m-c+1,m]$ either.
        Hence, parking fails.
    \end{proof}
    \begin{case}
         $\y'=(b,a^{k-1},c)$ with $a>c$.
    \end{case}
    \begin{proof} Here, car $k$, with length $a>c$, cannot park in $[m-c+1,m]$, so parking fails.
    \end{proof}
    Therefore, $c=a$, and $\y'=(b,a^k)$, so $y_1'=y_2'=\cdots=y_{k+1}'$, completing the induction.
    %(by our characterization of the invariant parking sequences for such $\y'$, there indeed exists $\mathbf{w}'$ with all entries greater than $1$ such that $(1,\mathbf{w}') \in \PAINV_{k+1}(\y')$), so $y'_2=\cdots=y'_{k+1}$, which completes the induction. 
\end{proof}
\begin{remark}
    The contrapositive proof strategy displayed above in the inductive step can be used to provide an alternate proof of Lemma \ref{lemma: charmaximpliesACbasecase}.
    Thus, Theorem \ref{theorem: charmaximpliesAC} can be proven without the use of Theorem 5.4 in ~\cite{icermpaper2023inv}; it can then be utilized to provide a quicker proof of that theorem since we have a necessary condition for the characteristic to be maximal.
\end{remark}
Putting everything together, we arrive at Theorem \ref{mainthm: maxchar}.
% \begin{corollary}
%     Let $\y=(y_1,y_2,\dots,y_n)\in \N^n$.
%     Then $\chi(\y)=n-1$ if and only if
%     \begin{equation} 
%     \label{eqn: equivmaxchar}
%         y_1\geq y_2 \quad \text{and} \quad y_2=y_3=\cdots=y_n
%     \end{equation}
%     %There exists $\x \in \PAINV_n(\y)$ with $\deg \x=n-1$ if and only if $y_2=y_3=\cdots=y_n$ and $y_1\geq y_2$.
% \end{corollary}
\begin{proof}[Proof of Theorem \ref{mainthm: maxchar}]
    If (\ref{eqn: equivmaxchar}) holds, then $\chi(\y)=n-1$ by Theorems \ref{theorem: ACmultiple}, \ref{theorem: AClarge}, and \ref{theorem: ACsmall}. 
    If $\chi(\y)=n-1$, then Theorem \ref{theorem: charmaximpliesAC} yields $y_2=y_3=\cdots=y_n$.
    Given this, Theorems \ref{theorem: ACmultiple}, \ref{theorem: AClarge}, and \ref{theorem: ACsmall} imply $y_1\geq y_2$, as desired.
    % The other direction follows from the proposition above.
\end{proof}
% \begin{remark}
%     Since we now have a direct characterization of characteristic $n-1$, this effectively reduces the search space for determining $\PAINV_n(\y)$ to $m^p$, where $p<n-1$.
% \end{remark}
To conclude this section, we present the enumerative results on $|\PAINV_n(\y)|$ and $|\PAINVND_n(\y)|$ when $\y$ is almost constant.
For constant $\y$, recall the following.
\begin{theorem}[Corollary~3.3 and Corollary~3.4 in \cite{icermpaper2023inv}]
\label{theorem: constantcounting}
    Let $\y=(a^n) \in \N^n$. 
    Then
    \[ |\PAINV_n(\y)|=(n+1)^{n-1} \quad \text{and} \quad |\PAINVND_n(\y)|=C_n, \]
    where $C_n=\frac{1}{n+1}\binom{2n}{n}$ is the $n$th Catalan number \footnote{OEIS \href{https://oeis.org/A000108}{A000108}.}.
\end{theorem}
As discussed, the exact same is true for certain almost constant $\y$.
\begin{corollary}
\label{corollary: easycounting}
    Let $\y=(b,a^{n-1}) \in \N^n$, where $n\geq 2$.
    If $a \mid b$ or $b>(n-1)a$, then 
    \[ |\PAINV_n(\y)|=(n+1)^{n-1} \quad \text{and} \quad |\PAINVND_n(\y)|=C_n. \]
\end{corollary}
\begin{proof}
    By Theorems \ref{theorem: constant}, \ref{theorem: ACmultiple}, and \ref{theorem: AClarge}, we have $\PAINV_n((b,a^{n-1}))=\PAINV_n((a^n))$.
    Thus, the result follows by Theorem \ref{theorem: constantcounting}.
\end{proof}
We now readily obtain a proof of Theorem \ref{mainthm: almostconstant}\ref{mainthm: almostconstantreg}.
\begin{proof}[Proof of Theorem \ref{mainthm: almostconstant}\ref{mainthm: almostconstantreg}]
    This is Theorems \ref{theorem: ACmultiple} and \ref{theorem: AClarge} and Corollary \ref{corollary: easycounting}. 
\end{proof}
To prove the next corollary, we will appeal to the following classical counting lemma due to Pitman and Stanley ~\cite{pitmanstanley2002polytope}.
For a nondecreasing $\mathbf{u}=(u_1,u_2,\dots,u_n) \in \N^n$ with successive differences $\Delta(\mathbf{u})\coloneqq (u_1,u_2-u_1,\dots,u_n-u_{n-1}) \in \N_0^n$, let $\PF_n(\mathbf{u})$ denote the set of $\mathbf{u}$-parking functions of length $n$, and let $\PFND_n(\mathbf{u}) \subseteq \PF_n(\mathbf{u})$ denote its subset of nondecreasing elements.
% \begin{definition}
%     Let $\mathbf{u}=(u_1,u_2,\dots,u_n) \in \N^n$ be nondecreasing.
%     Then $\x=(x_1,x_2,\dots,x_n) \in \N^n$ is a \emph{$\mathbf{u}$-parking function of length $n$} if $x_{(i)}\leq u_i$ for all $i \in [n]$.
%     We use $\PF_n(\mathbf{u})$ to denote the set of $\mathbf{u}$-parking functions of length $n$ and $\PFND_n(\mathbf{u}) \subseteq \PF_n(\mathbf{u})$ to denote its subset of nondecreasing elements.
%     For brevity, let $\Delta(\mathbf{u})$ denote the successive differences $(u_1,u_2-u_1,\dots,u_n-u_{n-1}) \in \N_0^n$.
% \end{definition}
\begin{lemma}[(8) in \cite{pitmanstanley2002polytope}]
\label{lemma: pitmanstanley}
    Let $\mathbf{u}=(u_1,u_2,\dots,u_n) \in \N^n$ be nondecreasing and $k \in [n-1]$.
    Assume that $\Delta(\mathbf{u})=(a,b^{n-k-1},c,0^{k-1})$.
    Then
    \[ |\PF_n(\mathbf{u})|=a\sum_{j=0}^{k}\binom{n}{j}(c-(k+1-j)b)^j(a+(n-j)b)^{n-j-1}. \]
\end{lemma}
Lemma \ref{lemma: pitmanstanley} arose from the study of empirical distributions and order statistics.
It was generalized by Yan in ~\cite{yan2000uparking} via two combinatorial proofs: one involved a strategic decomposition of a $\mathbf{u}$-parking function into subsequences, and the other was a recursive argument.

Armed with this result, we proceed to present and prove the second set of counting results.
\begin{corollary}
\label{corollary: weirdcounting}
    Let $\y=(b,a^{n-1}) \in \N^n$, where $n\geq 2$.
    If $a \nmid b$ and $b<(n-1)a$, then 
    \[ |\PAINV_n(\y)|=\sum_{j=0}^{n-\floor{b/a}-1}(-1)^j\binom{n}{j}\left( n-\floor{\frac{b}{a}}-j \right)^j(n-j+1)^{n-j-1} \]
    and
    \[ |\PAINVND_n(\y)|=\displaystyle\frac{n-\floor{b/a}+1}{n+1}\binom{n+\floor{b/a}}{\floor{b/a}}. \]
% \frac{\left( n-\floor{\frac{b}{a}}+1 \right)\left( n+\floor{\frac{b}{a}} \right)!}{\floor{\frac{b}{a}}!(n+1)!}
\end{corollary}
\begin{proof}
    Let $\x=(x_1,x_2,\dots,x_n) \in \PAINV_n(\y)$.
    We first compute $|\PAINV_n(\y)|$.
    % We first simplify the counting process for $\PAINV_n(\y)$ by scaling down the entries of $\x$. 
    The main claim is the following, which states that, for an appropriate choice of $\mathbf{u}$, such (resp. nondecreasing) invariant parking assortments are (resp. nondecreasing) $\mathbf{u}$-parking functions up to scaling.
    \begin{claim}
    \label{claim: simplifycount}
        Let $\mathbf{u}=\left(1,2,\dots,\floor{\frac{b}{a}},\left(\floor{\frac{b}{a}}+1\right)^{n-\floor{b/a}}\right)$. 
        Then $\PAINV_n(\y)$ and $\PF_n(\mathbf{u})$ are in bijection, and $\PAINVND_n(\y)$ and $\PFND_n(\mathbf{u})$ are in bijection.
        % Define $U_n(a,b)$ to be the set of all $\mathbf{p}=(p_1,p_2\dots,p_n) \in [n]^n$ satisfying the following condition:
        % \begin{equation} 
        % \label{eqn: simplifycount}
        %     p_{(i)}\leq \begin{cases}
        %     i & \forall i \in \left[ \floor{\frac{b}{a}}+1 \right] \\
        %     \floor{\frac{b}{a}}+1 & \text{otherwise}.
        % \end{cases} 
        % \end{equation}
        % Then $\PAINV_n(\y)$ and $U_n(a,b)$ are in bijection.
    \end{claim}
    \begin{proof}
        We construct a bijection between the sets in question.
        Define $\varphi:\PAINV_n(\y) \to \PF_n(\mathbf{u})$ by
        \[ \varphi(\x)=\left( 1+\frac{x_1-1}{a},1+\frac{x_2-1}{a},\dots,1+\frac{x_n-1}{a} \right)\coloneqq (\varphi(\x)_1,\varphi(\x)_2,\dots,\varphi(\x)_n). \]
        By Theorem \ref{theorem: ACsmall}, $x_{(i)}\leq 1+(i-1)a$ for all $i \in \left[\floor{\frac{b}{a}}+1\right]$, so $\varphi(\x)_{(i)}\leq i$ for all $i \in \left[\floor{\frac{b}{a}}+1\right]$.
        For the same reason, $x_{(i)}\leq 1+\floor{\frac{b}{a}}a$ for all $i \in \left\{ \floor{\frac{b}{a}}+1,\floor{\frac{b}{a}}+2,\dots,n \right\}$, so $\varphi(\x)_{(i)}\leq \floor{\frac{b}{a}}+1$ for all $i \in \left\{ \floor{\frac{b}{a}}+1,\floor{\frac{b}{a}}+2,\dots,n \right\}$.
        Hence, $\varphi(\x) \in \PF_n(\mathbf{u})$.
        Furthermore, for all $i \in [n]$, we have $x_i\equiv 1\pmod{a}$, so $\frac{x_i-1}{a} \in \mathbb{Z}$ for all $i \in [n]$, which implies that $\varphi(\x)_i \in \mathbb{Z}$.
        Putting everything together, since $1\leq x_i\leq 1+(n-1)a$, we have $1\leq \varphi(\x)_i\leq n$, so $\varphi(\x) \in \PF_n(\mathbf{u})$.
    
        We now construct the inverse of $\varphi$ to show that it is a bijection.
        Given $\mathbf{p}=(p_1,p_2,\dots,p_n) \in \PF_n(\mathbf{u})$, define $\psi: \PF_n(\mathbf{u}) \to \PAINV_n(\y)$ by
        \[ \psi(\mathbf{p})=(1+(p_1-1)a,1+(p_2-1)a,\dots,1+(p_n-1)a)\coloneqq \{ \psi(\mathbf{p})_1,\psi(\mathbf{p})_2,\dots,\psi(\mathbf{p})_n \}. \]
        Since $p_{(i)}\leq i$ for all $i \in \left[\floor{\frac{b}{a}}+1\right]$, we have $\psi(\x)_{(i)}\in \{ 1,1+a,\dots,1+(i-1)a \}$ for all $i \in \left[\floor{\frac{b}{a}}+1\right]$, and as $p_{(i)}\leq \floor{\frac{b}{a}}+1$ for all $i \in \left\{ \floor{\frac{b}{a}}+1,\floor{\frac{b}{a}}+2,\dots,n \right\}$, we have $\psi(\x)_{(i)}\in \left\{ 1,1+a,\dots,1+\floor{\frac{b}{a}}a \right\}$ for all $i \in \left\{ \floor{\frac{b}{a}}+1,\floor{\frac{b}{a}}+2,\dots,n \right\}$.
        Thus, $\psi(\mathbf{p}) \in \PAINV_n(\y)$ by Theorem \ref{theorem: ACsmall}.
        Note that $\psi \circ \varphi(\x)=\x$ and $\varphi \circ \psi(\mathbf{p})=\mathbf{p}$, so we indeed have $\varphi^{-1}=\psi$.

        Lastly, we show that $\varphi:\PAINVND_n(\y) \to \PFND_n(\mathbf{u})$ is also a bijection.
        It suffices to prove that $\varphi$ and $\psi$ map nondecreasing elements to nondecreasing elements.
        Note that if $\x=(x_1,x_2,\dots,x_n) \in \PAINVND_n(\y)$, then for all $i,j \in [n]$ with $i\leq j$, $x_i\leq x_j$ implies $1+\frac{x_i-1}{a}\leq 1+\frac{x_j-1}{a}$ since $a>0$, so $\varphi(\x) \in \PFND_n(\mathbf{u})$.
        Similarly, if $\mathbf{p}=(p_1,p_2,\dots,p_n) \in \PFND_n(\mathbf{u})$, then for all $i,j \in [n]$ with $i\leq j$, $p_i\leq p_j$ implies $1+(p_i-1)a\leq 1+(p_j-1)a$ since $a>0$, so $\psi(\mathbf{p}) \in \PAINVND_n(\y)$.
    \end{proof}
    Claim \ref{claim: simplifycount} establishes that it suffices to compute $|\PF_n(\mathbf{u})|$.
    To do so, we compute $\Delta(\mathbf{u})=(1^{\floor{b/a}+1},0^{n-\floor{b/a}-1})=(1,1^{\floor{b/a}},0,0^{n-\floor{b/a}-2})$.
    The result then follows by Lemma \ref{lemma: pitmanstanley}.

    Now, we turn our attention to computing $|\PAINVND_n(\y)|$.
    For $n,k \in \N_0$ such that $n\geq k$, let
    \[ f(n,k)\coloneqq |\PFND_n((1,2,\dots,k,(k+1)^{n-k}))|. \]
    By Claim \ref{claim: simplifycount}, $|\PAINVND_n(\y)|=f\left(n,\floor{\frac{b}{a}}\right)$, where $n>\floor{\frac{b}{a}}+1$, so it suffices to compute $f(n,k)$.
    To do so, we will make use of the following recursion.
    \begin{claim}\label{claim: CTrec}
        We have $f(n,0)=1$ for all $n \in \N_0$, $f(n,1)=n$ and $f(n+1,n+1)=f(n+1,n)$ for all $n \in \N$, and $f(n+1,k)=f(n+1,k-1)+f(n,k)$ for all $n,k\in \N$ such that $1<k<n+1$.
    \end{claim}
    \begin{proof}
        We have $f(n,0)=|\PFND_n((1^n))|=|\{ (1^n) \}|=1$ by definition \footnote{for $n=0$, the empty tuple is the unique element in this set, but this will be immaterial for our purposes.}.
        
        Similarly, $\PFND_n((1,2^{n-1}))=\{ (1^j,2^{n-j}):j \in [n] \}$ because if $\mathbf{p}=(p_1,p_2,\dots,p_n) \in \PFND_n((1,2^{n-1}))$, then there exists $j=\max\{ i\in [n]:p_i=1 \}$, and for $i \in [n] \setminus [j]$, we must have $p_i=2$, showing that $\mathbf{p}=(1^j,2^{n-j})$; the reverse inclusion is immediate.
        Thus, $f(n,1)=|\PFND_n((1,2^{n-1}))|=n$.
        Moreover, $f(n+1,n+1)=f(n+1,n)$ since 
        \[ (1,2,\dots,n+1,((n+1)+1)^{(n+1)-(n+1)})=(1,2,\dots,n+1)=(1,2,\dots,n,(n+1)^{(n+1)-n}). \]
        
        Now, for $n,k\in \N$ such that $1<k<n+1$, let $\mathbf{v}=(1,2,\dots,k,(k+1)^{n+1-k})$, so that $f(n+1,k)=|\PFND_{n+1}(\mathbf{v})|$, and let $\mathbf{v}'=(1,2,\dots,k-1,k^{n-k+2})$.
        Notice that
        \[ \PFND_{n+1}(\mathbf{v})=\PFND_{n+1}(\mathbf{v}') \cup \{ \mathbf{p}=(p_1,p_2,\dots,p_{n+1}) \in \PFND_{n+1}(\mathbf{v}): p_{n+1}=k+1 \}. \]
        It is immediate that the above is a disjoint union since if $\mathbf{p}=(p_1,p_2,\dots,p_{n+1}) \in \PFND_{n+1}(\mathbf{v}')$, then $p_{n+1}\leq k$.
        Next, let $\mathbf{p}=(p_1,p_2,\dots,p_{n+1}) \in \PFND_{n+1}(\mathbf{v})$.
        If $p_{n+1}<k+1$, then $\max \mathbf{p}\leq k$, so it must satisfy $p_i\leq i$ for all $i \in [k]$, and $p_i\leq k$ otherwise, meaning $\mathbf{p} \in \PFND_{n+1}(\mathbf{v}')$.
        The reverse inclusion follows by definition.
        Therefore, 
        \begin{align*} 
            f(n+1,k)&=|\PFND_{n+1}(\mathbf{v})| \\
            &=|\PFND_{n+1}(\mathbf{v}')|+|\{ \mathbf{p}=(p_1,p_2,\dots,p_{n+1}) \in \PFND_{n+1}(\mathbf{v}): p_{n+1}=k+1 \}| \\
            &=|\PFND_{n+1}(\mathbf{v}')|+|\{ \mathbf{p} \in \PFND_n(\mathbf{v''}): \mathbf{v}''=(1,2,\dots,k,(k+1)^{n-k}) \}| \\
            &=f(n+1,k-1)+f(n,k).
        \end{align*}
    \end{proof}
    Claim \ref{claim: CTrec} establishes that $f(n,k)$ satisfies the same recurrence as Catalan's triangle \footnote{OEIS \href{http://oeis.org/A009766}{A009766}.}, which has closed form $\frac{n-k+1}{n+1}\binom{n+k}{k}$ (cf. Lemma 1 and Theorem in ~\cite{bailey1996catalan}).
    Therefore, the result follows by setting $k=\floor{\frac{b}{a}}$.
    % To do so, we employ the principle of inclusion-exclusion.
    % By Definition \ref{defn: pforderstat}, we have $U_n(a,b) \subseteq \PF_n$, and note that $\PF_n \setminus U_n(a,b)=\left\{ \mathbf{p}=(p_1,p_2,\dots,p_n) \in \PF_n:p_{(n)}>\floor{\frac{b}{a}}+1 \right\}$.
    % Indeed, if $\mathbf{p} \in \PF_n \setminus U_n(a,b)$, then there exists an index $i$ for which $\floor{\frac{b}{a}}+1<i\leq n$ and $p_{(i)}>\floor{\frac{b}{a}}+1$, which implies $p_{(n)}>\floor{\frac{b}{a}}+1$ (the reverse inclusion follows by definition).
\end{proof}
\begin{remark}
    For $\floor{\frac{b}{a}}=0,1,2$, one can show using Corollary \ref{corollary: weirdcounting} that we have $|\PAINV_n(\y)|=1,2^n-1,3^n-2^n-n$, respectively. 
    As detailed by Yan in ~\cite{yan2000uparking}, closed forms of the sum in Lemma \ref{lemma: pitmanstanley} do not exist for general $n$.
    However, one fact worth noting is that when $\floor{\frac{b}{a}}=2$, it turns out that $|\PAINV_n(\y)|$ is also the number of regions in the $G$-Shi arrangement when $G$ is the cycle graph with $n$ vertices \footnote{OEIS \href{http://oeis.org/A355645}{A355645}}.
\end{remark}
To finish, we may now easily prove the rest of Theorem \ref{mainthm: almostconstant}.
\begin{proof}[Proof of Theorem \ref{mainthm: almostconstant}\ref{mainthm: almostconstantirreg}]
    This is Theorem \ref{theorem: ACsmall} and Corollary \ref{corollary: weirdcounting}.
\end{proof}
\section{Properties of $\PAINV_n(\y)$ in Relation to the Degree and Characteristic}
\label{section: deg and char}
% In this section, we continue to study what the characteristic of a length vector $\y \in \N^n$ reveals about its elements.
% We also examine how the characteristic influences the behavior of $\PAINV_n(\y)$ and the possible degrees of its elements.
Given the intricacies of the set $\PAINV_n(\y)$ for $\y$ of maximal characteristic, our aim in this section is to study the structure of the set given any $\y$.
As we will see, deriving certain structural results can allow us to recover general information about the degree and characteristic; this will comprise Theorem \ref{mainthm: degchar}.
% We also continue to examine the inverse problem of determining entries of $\y$ given its characteristic.
% To the former end, we provide two general results: one being a necessary condition for characteristic $0$ and another concerning invariant parking assortments for positive characteristic.
% To the latter end, we present two main results.
% Namely, we show that if $\x \in \PAINV_n(\y)$, then $\x$ is still an invariant parking assortment for the length vector obtained by appending entries to $\y$.
% In particular, the characteristic of this vector is either $\chi(\y)$ or $\chi(\y)+1$. 
% Moreover, we prove that invariant parking assortments for $\y$ are closed under replacement with $1$s.
% Consequently, if $\chi(\y)=\alpha$, then the image of $\deg: \PAINV_n(\y) \to [n-1]_0$ is precisely $[\alpha]_0$; which reveals that the characteristic encodes all attainable degrees for $\x \in \PAINV_n(\y)$. 
%As we will see, these two properties are intimately connected, where the former gives a quick proof of the latter.
%This property is conveyed in the following lemma:
%\begin{lemma}
    %Let $\y=(y_1,y_2,\dots,y_n) \in \N^n$ and $\x \in \PAINV_n(\y)$, where $\deg \x=d$. Then there exists $\x' \in \PAINV_n(\y)$ such that $\deg \x'=d-1$.
%\end{lemma}
%\begin{proof}
    %We are given that $\x$ is of the form $(1^{n-d},p_1,p_2,\dots,p_d)$, where $p_1,p_2,\dots,p_d>1$. Without loss of generality, assume $\x$ is nondecreasing. 
%\end{proof}

Throughout this section, the following result will be helpful.
\begin{lemma}[Proposition 2.2 in ~\cite{icermpaper2023inv}]
    \label{lemma: nondecreasing}
    Let $\y = (y_1, y_2, \ldots, y_n) \in \N^n$. 
    Assume $\x = (x_1, x_2, \ldots, x_n) \in [m]^n$ is nondecreasing.
    Then, $\x \in \PA_n(\y)$ if and only if $x_i \leq 1 + \sum_{j=1}^{i-1} y_j$ for all $i \in [n]$.
\end{lemma}
% \begin{lemma}[Corollary 4.7 in ~\cite{icermpaper2023inv}]
%     Let $\y = (y_1, y_2, \ldots, y_n) \in \N^n$. 
%     If $\chi(\y)=0$, then $\chi(\y_{\vert_i})=0$ for all $i \in [n]$.
% \end{lemma}
For proof of Theorem \ref{mainthm: degchar} specifically, 
% we consider appending $y_{n+1} \in \N$ to $\y$ to form $\y^+\in \N^{n+1}$, and we show that there are exactly two possibilities for $\chi(\y^+)$. 
the three lemmas below will be crucial ingredients; Lemma \ref{lemma: replacement} is due to~\cite{personalcommuncation}, and we provide an independent proof.
In particular, Lemma \ref{lemma: swapalgo} will allow us to morph an invariant parking assortment into a more convenient parking assortment, which yields a useful technique to help prove Theorem \ref{mainthm: degchar}\ref{mainthm: degchar closure}.
\begin{lemma}
\label{lemma: extendPA}
    Let $\y=(y_1,y_2,\dots,y_n) \in \N^n$ and $\y^+=(\y,y_{n+1}) \in \N^n$. 
    If $\x \in \PA_n(\y)$, then $\x^+=(\x,x_{n+1}) \in \PA_{n+1}(\y^+)$ if and only if $x_{n+1}\leq 1+\sum_{j=1}^{n}y_j$. 
\end{lemma}
\begin{proof}
    Consider the parking experiment under $\x^+$. 
    Because $\x \in \PA_n(\y)$, the first $n$ cars occupy $[1,\sum_{j=1}^{n}y_j]$. 
    This leaves only spots $[1+\sum_{j=1}^{n}y_j,\sum_{j=1}^{n+1}y_j]$ empty. 
    If $x_{n+1}>1+\sum_{j=1}^{n}y_j$, then the last car, with length $y_{n+1}$, must park in $[x_{n+1},\sum_{j=1}^{n+1}y_j]$, but $|[x_{n+1},\sum_{j=1}^{n+1}y_j]|<y_{n+1}$, so parking fails.
    If $x_{n+1}\leq 1+\sum_{j=1}^{n}y_j$, then since $[1,\sum_{j=1}^{n}y_j]$ is occupied, the last car occupies $[1+\sum_{j=1}^{n}y_j,\sum_{j=1}^{n+1}y_j]$, so parking succeeds.
    % so parking succeeds since $|[x_{n+1},\sum_{j=1}^{n+1}y_j]|=y_{n+1}$.
\end{proof}
\begin{lemma}
\label{lemma: replacement}
    Let $\y=(y_1,y_2,\dots,y_n) \in \N^n$. Then if $\x=(x_1,x_2,\dots,x_n) \in \PA_n(\y)$, and there exists $i \in [n]$ such that $x_i=\min(x_i,x_{i+1},\dots,x_n)$, then $\mathbf{r}=(x_1,\dots,x_{i-1},r,x_{i+1},\dots,x_n) \in \PA_n(\y)$ for any $r \in [x_i]$. In particular, $(\x_{\vert_{n-1}},1) \in \PA_n(\y)$.
\end{lemma}
\begin{proof}
    Consider the parking experiment under $\x$. 
    The key observation is that $[1,x_i-1]$ is occupied when car $i$ attempts to park. 
    To see this, assume the contrary; i.e. there exists an $s \in [1,x_i-1]$ that is unoccupied.
    Then because $x_i=\min(x_i,x_{i+1},\dots,x_n)$, cars $i,i+1,\dots,n$ all drive past $s$, leaving $s$ unoccupied after all cars have parked, contradicting $\x \in \PA_n(\y)$. 
    
    Thus, if car $i$ instead had preference less than $x_i$ (and all other preferences are unchanged), then since $[1,x_i-1]$ is occupied, its choice of parking spots remains the same, implying that all cars can still park. 
    Therefore, $\mathbf{r} \in \PA_n(\y)$.
\end{proof}
\begin{lemma}
\label{lemma: swapalgo}
    Let $\mathbf{a}=(a_1,a_2,\dots,a_n) \in \N^n$ and $S=\{ i \in [n]:a_1>a_i \}$. 
    Then there is a permutation $\b=(b_1,b_2,\dots,b_n)$ of the entries of $\mathbf{a}$ such that 
    \begin{enumerate}[label=(\roman*)]
        \item\label{lemma: swapalgo i} $b_1=\min \b$.
        \item\label{lemma: swapalgo ii} $b_i\geq a_i$ for all $i \in S$ and $b_i=a_i$ for all $i \notin S \cup \{ 1 \}$.
        \item\label{lemma: swapalgo iii} $b_i=\min(b_i,b_{i+1},\dots,b_n)$ for all $i \in S$ such that $b_i>a_i$.
    \end{enumerate}
\end{lemma}
\begin{proof}
    We will describe an algorithm that sequentially swaps certain entries of $\mathbf{a}$ to construct $\b$. 
    Set $\mathbf{a}^{(0)}=\mathbf{a}$ and $S^{(0)}=S$. 
    For $k \in \N_0$, first check if $\mathbf{a}^{(k)}=(a_1^{(k)},a_2^{(k)},\dots,a_n^{(k)})$ satisfies $a_1^{(k)}=\min\mathbf{a}^{(k)}$.
    If so, we claim $\b=\mathbf{a}^{(k)}$ and stop. 
    Otherwise, construct the index set $S^{(k)}=\{ i \in [n]:a_1^{(k)}>a_i^{(k)} \}$, and let $j_k=\max S^{(k)}$. 
    Then, swap the positions of $a_1^{(k)}$ and $a_{j_k}^{(k)}$ to obtain the next iterate
    \[ \mathbf{a}^{(k+1)}=(a_1^{(k+1)},a_2^{(k+1)},\dots,a_n^{(k+1)})\coloneqq (a_{j_k}^{(k)},a_2^{(k)},\dots,a_{j_k-1}^{(k)},a_1^{(k)},a_{j_k+1}^{(k)},\dots,a_n^{(k)}), \]
    and repeat the process.

    To prove the correctness of this algorithm, we will first show that it terminates; this gives property \ref{lemma: swapalgo i} in the process.
    \begin{claim}
    \label{claim: indexnestsd}
        The sequence of index sets $\{ S^{(k)} \}$ is nested and strictly decreasing.
    \end{claim}
    \begin{proof}
        Let $k$ be any nonnegative integer, and assume $S^{(k)}\neq \emptyset$.
        Because $a_1^{(k+1)}=a_{j_{k}}^{(k)}$, where $j_{k}=\max S^{(k)}=\max\{ i \in [n]:a_1^{(k)}>a_i^{(k)} \}$, we have $a_1^{(k)}>a_1^{(k+1)}$ (i.e. the sequence consisting of the first entries $a_1^{(k)}$ of the iterates is strictly decreasing).
        Moreover, by construction, $a_i^{(k)}=a_i^{(k+1)}$ if and only if $i\neq 1,j_{k}$. 
        It is clear that $1\notin S^{(k)}\cup S^{(k+1)}$.
        We also have $a_{j_{k}}^{(k+1)}=a_1^{(k)}>a_1^{(k+1)}$, so $j_{k}\notin S^{(k+1)}$. 
        Thus,
        \[ S^{(k+1)}=\{ i \in [n]:a_1^{(k+1)}>a_i^{(k+1)} \}=\{ i \in [n]:a_1^{(k+1)}>a_i^{(k)} \}\subsetneq \{ i \in [n]:a_1^{(k)}>a_i^{(k)} \}=S^{(k)}, \]
        as $a_1^{(k)}>a_1^{(k+1)}$ and $j_{k} \in S^{(k)}$, while $j_{k} \notin S^{(k+1)}$, implying our claim.
    \end{proof} 

    \begin{claim}
    \label{claim: indexemptyequiv}
        We have 
        \[ \argmin \mathbf{a} \notin S^{(k)} \iff \mathbf{a}^{(k)} \text{ satisfies } a_1^{(k)}=\min\mathbf{a} \iff S^{(k)}=\emptyset. \]
    \end{claim}
    \begin{proof}
        Note that $\min\mathbf{a}=\min\mathbf{a}^{(k)}$ since each iterate is a permutation of the entries of $\mathbf{a}$ by construction. If $\argmin \mathbf{a} \notin S^{(k)}$, then all $i \in [n]$ such that $a_1^{(k)}>a_i^{(k)}$ satisfy $a_i^{(k)}>\min \mathbf{a}$. 
        But this means $\min \mathbf{a}\neq a_i^{(k)}$ for all $i \in [n] \setminus \{ 1 \}$, so $a_1^{(k)}=\min\mathbf{a}$. 
        Next, if $\mathbf{a}^{(k)}$ satisfies $a_1^{(k)}=\min \mathbf{a}$, then no $i \in [n]$ can satisfy $a_1^{(k)}>a_i^{(k)}$, so $S^{(k)}=\emptyset$. 
        Lastly, if $S^{(k)}=\emptyset$, then it is clear $\argmin \mathbf{a} \notin S^{(k)}$, proving the equivalence.   
    \end{proof}
    
    Combining Claims \ref{claim: indexnestsd} and \ref{claim: indexemptyequiv}, the algorithm must terminate because the index sets are nested and strictly decreasing, so there exists $L$ for which $S^{(L)}=\emptyset$, which is equivalent to $\mathbf{a}^{(L)}$ satisfying $a_1^{(L)}=\min\mathbf{a}$, or property \ref{lemma: swapalgo i}. It now suffices to show that $\mathbf{a}^{(L)}$ satisfies properties \ref{lemma: swapalgo ii} and \ref{lemma: swapalgo iii}.
    
    Now, we note that by Claim \ref{claim: indexnestsd}, if $i \notin S^{(k)}$, then $i \notin S^{(\ell)}$ for all $\ell>k$.
    In particular, as $a_i^{(k)}=a_i^{(k+1)}$ if and only if $i\neq 1,j_k$, it follows that if $i\notin S \cup \{ 1 \}$, then $a_i^{(k)}=a_i^{(k+1)}$ for all $k$, which shows $a_i^{(L)}=a_i$. 
    Moreover, if $i \notin S^{(k)} \cup \{ 1 \}$, then $a_i^{(k)}=a_i^{(k+1)}=\cdots=a_i^{(L)}$.
    
    Similarly, if $i \in S$, then there exists $k$ such that $i \in S^{(k)}$ but $i \notin S^{(k+1)}$. 
    Note that $i \in S^{(k)} \subsetneq S^{(k-1)} \subsetneq \cdots \subsetneq S^{(0)}$. 
    Then $i\neq j_0,j_1,\dots,j_{k-1}$; if not, then $i=j_\ell$ for some $0\leq\ell\leq k-1$, and $j_\ell \notin S^{(\ell+1)}$, which implies $j_\ell \notin S^{(k)}$.
    Hence, $a_i^{(0)}=a_i^{(1)}=\cdots=a_i^{(k)}$. 
    At this stage of the algorithm, we have two possibilities: either $i=j_k=\max S^{(k)}$ or $i<j_k$ with $a_i^{(k)}\geq a_{j_k}^{(k)}=a_1^{(k+1)}$. 
    For the former, as $i=j_k \notin S^{(k+1)} \cup \{ 1 \}$, we have $a_1^{(k)}=a_{j_k}^{(k+1)}=\cdots=a_{j_k}^{(L)}$, where $a_1^{(k)}>a_{j_k}^{(k)}$; hence, $a_{j_k}^{(L)}>a_{j_k}^{(0)}$. 
    For the latter, $i\neq j_0,j_1,\dots,j_{L-1}$, as $i\notin S^{(k+1)}$ and hence $i\notin S^{(\ell)}$ for any $\ell>k$, so $a_i^{(0)}=a_i^{(1)}=\cdots=a_i^{(L)}$. 
    Thus, $a_i^{(L)}\geq a_i$ for all $i \in S$.
    Altogether, we obtain property \ref{lemma: swapalgo ii}.

    Lastly, assume that $i \in S$ satisfies $a_i^{(L)}>a_i$. 
    We showed above that if $i\neq j_k$ for any $k$, then $a_i^{(L)}=a_i$. 
    Thus, we must have $i=j_k \notin S^{(k+1)}$ for some $k$, so $a_1^{(k)}=a_{j_k}^{(k+1)}=a_{j_k}^{(L)}$. 
    By the definition of $j_k$ and $\mathbf{a}^{(k+1)}$, we have
    \[ a_1^{(k+1)}=a_{j_k}^{(k)}<a_1^{(k)}\leq a_{j_k+1}^{(k)},\dots,a_n^{(k)}=a_{j_k+1}^{(k+1)},\dots,a_n^{(k+1)}, \]
    which yields $j_k+1,\dots,n \notin S^{(k+1)}$, so $(a_{j_k}^{(k+1)},a_{j_k+1}^{(k+1)},\dots,a_n^{(k+1)})=(a_{j_k}^{(L)},a_{j_k+1}^{(L)},\dots,a_n^{(L)})$.
    Consequently, $a_{j_k}^{(L)}=a_1^{(k)}=\min(a_{j_k}^{(L)},a_{j_k+1}^{(L)},\dots,a_n^{(L)})$, which proves property \ref{lemma: swapalgo iii}.

    Therefore, we may take $\b=\mathbf{a}^{(L)}$, as desired.
\end{proof}
We can now leverage these lemmas to prove Theorem \ref{mainthm: degchar} successively (in order, we prove the closure property, image of the degree, embedding property, and monotonicity).
% prove the theorems below, 
% which imply the closure and embedding properties for $\PAINV_n(\y)$.
% \begin{theorem}
% \label{theorem: closrep1}
%     Let $\y=(y_1,y_2,\dots,y_n) \in \N^n$. If $(1^{n-d},\mathbf{w}) \in \PAINV_n(\y)$, where $\mathbf{w} \in \N^d_{>1}$, then $(1^{n-d+1},\mathbf{w}_{\widehat{i}}) \in \PAINV_n(\y)$ for all $i \in [d]$.
% \end{theorem}
\begin{proof}[Proof of Theorem \ref{mainthm: degchar}\ref{mainthm: degchar closure}]
    For any $i \in [d]$, the main idea is to start with a permutation $\mathbf{p}=(p_1,p_2,\dots,p_n)$ of $(1^{n-d+1},\mathbf{w}_{\widehat{i}})$, construct a specific permutation $\x$ of $(1^{n-d},\mathbf{w}) \in \PAINV_n(\y)$, and use Lemmas \ref{lemma: replacement} and \ref{lemma: swapalgo} to progressively morph $\x$ into $\mathbf{p}$ while ensuring that each change preserves membership in $\PA_n(\y)$.
    
    Note first that if $p_n=1$, then $\mathbf{p}_{\vert_{n-1}}$ is a permutation of $(1^{n-d},\mathbf{w}_{\widehat{i}})$, so $(\mathbf{p}_{\vert_{n-1}},w_i) \in \PA_n(\y)$. 
    Applying Lemma \ref{lemma: replacement}, we obtain $(\mathbf{p}_{\vert_{n-1}},1)=\mathbf{p} \in \PA_n(\y)$.
    
    Now, we suppose $p_n=v \in \mathbf{w}_{\widehat{i}}$. 
    Assume first that $v\leq w_i$. Let $k=\max\{ j \in [n]:p_j=1 \}$. Consider the following permutation of $(1^{n-d},\mathbf{w})$:
    \[ \x=(p_1,\dots,p_{k-1},v,p_{k+1},\dots,p_{n-1},w_i) \in \PA_n(\y). \]
    Let $\mathbf{q}=(q_1,q_2,\dots,q_{n-k+1})\coloneqq (v,p_{k+1},\dots,p_{n-1},w_i)$ and $S=\{ j \in [n-k+1]:q_1>q_j \}$. 
    By Lemma \ref{lemma: swapalgo}, there is a permutation $\mathbf{q}'=(q_1',q_2',\dots,q_{n-k+1}')$ of the entries of $\mathbf{q}$ such that $q_1'=\min \mathbf{q}'$, $q_j'\geq q_j$ for all $j \in S$, $q_j'=q_j$ for all $j \notin S \cup \{ 1 \}$, and $q_j'=\min(q_j',q_{j+1}',\dots,q_{n-k+1}')$ for all $j \in S$ such that $q_j'>q_j$. 
    Then $\mathbf{x}'=(\mathbf{x}_{\vert_{k-1}},\mathbf{q}')=(\mathbf{p}_{\vert_{k-1}},\mathbf{q}') \in \PA_n(\y)$ is a permutation of $\x$. 
    
    From here, we aim to transform $\mathbf{q}'$ into $(p_k,p_{k+1}\dots,p_n)$ by way of Lemma \ref{lemma: replacement}. 
    Since $q_1'=\min \mathbf{q}'$, Lemma \ref{lemma: replacement} yields $(\mathbf{p}_{\vert_{k-1}},1,q_2',\dots,q_{n-k+1}') \in \PA_n(\y)$.
    For $2\leq j<n-k+1$, if $j \notin S$, then $q_j'=q_j=p_{j+k-1}$. 
    If $j \in S$, then we have $q_j'\geq q_j=p_{j+k-1}$. 
    Consider the sequence
    \[ (h_1,h_2,\dots,h_\ell)\coloneqq (j \in S:q_j'>q_j), \] 
    where $h_1<h_2<\cdots<h_\ell$. 
    For any $t \in [\ell]$, inductively assuming that we have already replaced $q_{h_1},\dots,q_{h_{t-1}}$, we have $q_{h_t}'>q_{h_t}$, and $q_{h_t}'=\min(q_{h_t}',q_{h_t+1}',\dots,q_{n-k+1}')$, so applying Lemma \ref{lemma: replacement}, $q_{h_t}'$ may be replaced with $q_{h_t}=p_{h_t+k-1}$ to obtain a parking assortment \footnote{the first such replacement is justified for the same reason.}. 
    Lastly, as $n-k+1 \notin S$, we have $q_{n-k+1}'=q_{n-k+1}=w_i$. 
    Consequently, $(\mathbf{p}_{\vert_{n-1}},w_i) \in \PA_n(\y)$. 
    Applying Lemma \ref{lemma: replacement} a final time, as $v\leq w_i$, we have $(\mathbf{p}_{\vert_{n-1}},v)=\mathbf{p} \in \PA_n(\y)$. Now, assume that $v>w_i$. We will instead consider the following permutation of $(1^{n-d},\mathbf{w})$:
    \[ \x=(p_1,\dots,p_{k-1},w_i,p_{k+1},\dots,p_{n-1},v) \in \PA_n(\y). \]
    Repeating the same argument as in the case of $v\leq w_i$ yields $\mathbf{p} \in \PA_n(\y)$.

    This covers all permutations of $(1^{n-d+1},w_{\widehat{i}})$, so therefore, as $i \in [d]$ was arbitrary, we have $(1^{n-d+1},w_{\widehat{i}}) \in \PAINV_n(\y)$ for all $i \in [d]$, as desired.
\end{proof}
% \begin{corollary}
% \label{corollary: closure}
%     Let $\y \in \N^n$. 
%     Then the image of $\deg:\PAINV_n(\y) \to [n-1]_0$ is $[\chi(\y)]_0$.
%     % If $\x \in \PAINV_n(\y)$, where $\deg \x=d$, then for any $\delta \in [d-1]$, there exists $\x' \in \PAINV_n(\y)$ such that $\deg\x'=\delta$.
% \end{corollary}
\begin{proof}[Proof of Theorem \ref{mainthm: degchar}\ref{mainthm: degchar image}]
    We repeatedly apply Theorem \ref{mainthm: degchar}\ref{mainthm: degchar closure}.
\end{proof}
% \begin{theorem}
% \label{theorem: extension}
%     Let $\y=(y_1,y_2,\dots,y_n) \in \N^n$ and $\y^+=(\y,y_{n+1}) \in \N^{n+1}$. 
%     Then if $\x \in \PAINV_n(\y)$, then $(1,\x) \in \PAINV_{n+1}(\y^+)$.
%     In particular, we have the embedding
%     \[ \eta: \begin{cases}
%         \PAINVND_n(\y) & \hookrightarrow \PAINVND_{n+1}(\y^+) \\
%         \x & \mapsto (1,\x).
%     \end{cases} \]
% \end{theorem}
\begin{proof}[Proof of Theorem \ref{mainthm: degchar}\ref{mainthm: degchar embedding}]
    Assume that $\deg \x=d$. Let $\mathbf{p}=(p_1,p_2,\dots,p_{n+1})$ be any permutation of $(1,\x)$. 
    First, if $p_{n+1}=1$, then $\mathbf{p}_{\vert_n}$ is a permutation of $\x \in \PA_n(\y)$, so as $p_{n+1}\leq 1+\sum_{j=1}^{n}y_j$, we have $\mathbf{p} \in \PA_{n+1}(\y^+)$ by Lemma \ref{lemma: extendPA}. 
    Next, without loss of generality, $\x=(1^{n-d},\mathbf{w})$, where $\mathbf{w} \in \N_{>1}^d$. 
    If $p_{n+1}=w_i$ for some $i \in [d]$, then $\mathbf{p}_{\vert_n}$ is a permutation of $(1^{n-d+1},\mathbf{w}_{\widehat{i}}) \in \PAINV_n(\y)$ by Theorem \ref{mainthm: degchar}\ref{mainthm: degchar closure}. 
    Because $w_i\leq \max\mathbf{w}\leq 1+\sum_{j=1}^{n-1}y_j\leq 1+\sum_{j=1}^{n}y_j$ by Lemma \ref{lemma: nondecreasing}, we have $\mathbf{p} \in \PA_{n+1}(\y^+)$ by Lemma \ref{lemma: extendPA}. 
    This covers all permutations of $(1,\x)$; therefore, $(1,\x) \in \PAINV_{n+1}(\y^+)$.
\end{proof}
\begin{remark}
    This result immediately implies that the converse of Theorem \ref{mainthm: neccminchar} does not hold. 
    For instance, if $\y=(1,3,2,2)$, then one can check that $(1^3,4) \in \PAINV_4(\y)$, so $(1^{k+3},4) \in \PAINV_{k+4}(\y^+)$ for any extension $\y^+=(\y,y_5,\dots,y_{k+4})$ of $\y$.
    % , where $y_{n+1},\dots,y_{n+k}>1$.
\end{remark}
% As a consequence, Theorems \ref{theorem: closrep1} and \ref{theorem: extension} imply that there are precisely two options for $\chi(\y^+)$, which is stated in the result below.
% \begin{corollary}
% \label{corollary: charuplowbound}
%     Let $\y=(y_1,y_2,\dots,y_n) \in \N^n$, where $\chi(\y)=\alpha$, and $\y^+=(\y,y_{n+1}) \in \N^{n+1}$. 
%     Then $\chi(\y^+) \in \{ \alpha,\alpha+1 \}$.
% \end{corollary}
\begin{proof}[Proof of Theorem \ref{mainthm: degchar}\ref{mainthm: degchar prefix}]
    We first deal with the upper bound. 
    It is attainable, as if $\y=(c^n)$ and $\y^+=(c^{n+1})$, then $\chi(\y)=n-1$ and $\chi(\y^+)=n$. 
    To see that the upper bound is valid, suppose $\chi(\y^+)>\alpha+1$. 
    By Lemma \ref{lemma: removal}, $\chi(\y)>\alpha$. 
    But by assumption, $\chi(\y)=\alpha$, a contradiction. 
    Thus, $\chi(\y^+)\leq \alpha+1$.

    The lower bound is attainable, as if $\y$ is strictly increasing, and $\y^+=(\y,y_{n+1})$, where $y_{n+1}>y_n$, then $\chi(\y)=\chi(\y^+)=0$ (cf. Theorem 4.9 in ~\cite{icermpaper2023inv}). 
    By assumption, as $\chi(\y)=\alpha$, there exists $\x\in\PAINV_n(\y)$ such that $\deg \x=\alpha$. 
    Then by Theorem \ref{mainthm: degchar}\ref{mainthm: degchar embedding}, we have $(1,\x) \in \PAINV_{n+1}(\y^+)$, and $\deg((1,\x))=\alpha$, which means that $\chi(\y^+)\geq \alpha$.
\end{proof}
\begin{remark}
    This result generalizes Corollary 4.3 in \cite{icermpaper2023inv}, which is the case $\alpha=0$.
\end{remark}
\section{The Invariant Solution Set, Sum Avoidance, and Extremality}
\label{section: sdlengths}
In this section, we introduce the invariant solution set of $\y \in \N^n$ and relate it to $\PAINV_n(\y)$ and consider a family of strictly decreasing length vectors.
% This in turn will lead us to connections between the forms of invariant parking assortments, partial sums of the entries of $\y$, and $\chi(\y)$, as well as two extremal results.
This in turn will lead us to proofs of Theorems \ref{mainthm: invsolset} and \ref{mainthm: PAINVNDbound}.
% Our first definition will set up a foundation for what we will discuss.
We define the invariant solution set as follows.
\begin{definition}
    Let $\y \in \N^n$. 
    The \emph{invariant solution set of $\y$} is given by
    \[ \W(\y)\coloneqq \{ w \in \N:(1^{n-1},w) \in \PAINV_n(\y) \}. \]
\end{definition}
To illustrate why this set is useful to study, we have the following result, which is a consequence of Theorem \ref{mainthm: degchar}\ref{mainthm: degchar closure}. 
\begin{lemma}
\label{lemma: invsolset}
    If $\x \in \PAINV_n(\y)$, then $\x \in \W(\y)^n$. 
\end{lemma}
\begin{proof}
    % This follows trivially when $\deg \x=0$, so assume otherwise. 
    It is clear that $1\in \W(\y)$ for any $\y$, so consider any $i \in [n]$ such that $x_i\neq 1$. 
    Repeatedly applying Theorem \ref{mainthm: degchar}\ref{mainthm: degchar closure}, we have $(1^{n-1},x_i) \in \PAINV_n(\y)$, so $x_i \in \W(\y)$, as needed. 
\end{proof}
Given this, we now examine the elements of $\W(\y)$. 
Observe that Lemma \ref{lemma: nondecreasing} implies that if $w \in \W(\y)$, then $w\leq 1+\sum_{i=1}^{n-1}y_i$.
Interestingly, this bound can only be improved slightly if $\y$ is non-constant, which is Theorem \ref{mainthm: invsolset}\ref{mainthm: invsolset nonconstant}.
\begin{proof}[Proof of Theorem \ref{mainthm: invsolset}\ref{mainthm: invsolset nonconstant}]
    Again, we prove the contrapositive: if $w=1+\sum_{i=1}^{n-1}y_i$ satisfies $(1^{n-1},w) \in \PAINV_{n}(\y)$, then $\y$ is constant. 

    By assumption, $(1^{n-p-1},w,1^p)$ for all $p \in [n-1]_0$. 
    Fix $p$, and consider the parking experiment under $(1^{n-1},w)$.
    Then $[1,\sum_{i=1}^{n-p-1}y_i]$ is occupied before the $(n-p)$th car parks. 
    After this car parks, note that $\mathcal{U}_1\coloneqq [1+\sum_{i=1}^{n-p-1}y_i,w-1]$ 
    % where $|\mathcal{U}_1|=\sum_{i=1}^{p}y_{n-i}>\sum_{i=1}^{p-1}y_{n-i}>0$, 
    is unoccupied.
    Once the $(n-p)$th car parks, it fills $[1+\sum_{i=1}^{n-1}y_i,y_{n-p}+\sum_{i=1}^{n-1}y_i]$.
    Then $\mathcal{U}_2\coloneqq [1+y_{n-p}+\sum_{i=1}^{n-1}y_i,m]=[w+y_{n-p},m]$, where $|\mathcal{U}_2|=y_n-y_{n-p}\geq 0$, is unoccupied.
    Cars $n-p+1,n-p+2\dots,n-1$, all with preference $1$, will then park in $\mathcal{U}_1$, so that $\mathcal{U}_1'=[w-y_{n-p},w-1]$, where $|\mathcal{U}_1'|=y_{n-p}>0$, is still unoccupied. 
    Since $w-1<w+y_{n-p}$, the intervals $\mathcal{U}_1'$ and $\mathcal{U}_2$ are not contiguous, and the $n$th car must fill these intervals of lengths $y_{n-p}$ and $y_n-y_{n-p}$, which can only happen if $y_n=y_{n-p}$. 
    Therefore, because $p$ was arbitrary, $y_n=y_{n-p}$ for any $p \in [n-1]_0$, which implies that $\y$ is constant.
\end{proof}
\begin{remark}
    Theorem \ref{mainthm: invsolset}\ref{mainthm: invsolset nonconstant} also yields an alternate characterization of constant length vectors: $\y$ is constant if and only if $(1^{n-1},1+\sum_{i=1}^{n-1}y_i) \in \PAINV_n(\y)$.
    
    For an example of the equality case for $w$ when $\y$ is non-constant, consider $\y=(1,3,3,2)$. 
    One can check that $(1^3,7)=(1,y_1+y_2+y_3) \in \PAINV_4(\y)$.
    More generally, one also can show that if $\y=(1,3^{n-2},2) \in \N^n$, then $(1^{n-1},3n-5) \in \PAINV_n(\y)$.
\end{remark}
Moreover, the following result guarantees that any invariant solution can be written as a binary combination of the last $n-1$ entries of $\y$, which will allow us to prove Theorem \ref{mainthm: invsolset}\ref{mainthm: invsolset size bound}.
\begin{lemma}
\label{lemma: bincomb}
    Let $\y\in \N^n$ and $w \in \W(\y)$. 
    Then $w=1+\b^\top\y$ for some $\b \in \{ 0 \} \times \{ 0,1 \}^{n-1}$.
\end{lemma}
\begin{proof}
    The result is immediate for $w=1$, so assume $w>1$.
    We first show that there exists a $\b \in \{ 0,1 \}^n$ for which $w=1+\b^\top\y$.
    Assume not.
    Then $(w,1^{n-1}) \notin \PA_n(\y)$. 
    Indeed, $w-1$ is not a sum of entries of $\y$, meaning that no subset of the cars can precisely fill $[1,w-1]$, so $(1^{n-1},w) \notin \PAINV_n(\y)$. 
    Thus, such a $\b \in \{ 0,1 \}^n$ must exist.
    
    Now, if $w=1+\b^\top\y$, where $\b \in \{ 1 \} \times \{ 0,1 \}^{n-1}$, then we claim there exists $\mathbf{c} \in \{ 0 \} \times \{ 0,1 \}^{n-1}$ such that $\b^\top\y=\mathbf{c}^\top\y$. 
    Indeed, as $(w,1^{n-1}) \in \PA_n(\y)$, under its parking experiment, the first car leaves $[1,w-1]$ empty, and this must be completely occupied by the end of the experiment. 
    Thus, there is a subset of cars $2,3,\dots,n$ that can precisely fill $[1,w-1]$, which implies the existence of $\mathbf{c} \in \{ 0 \} \times \{ 0,1 \}^{n-1}$ such that $\mathbf{c}^\top\y=w-1=\b^\top\y$, as claimed.
    
    Therefore, for any $w \in \W(\y)$, we can always find $\b \in \{ 0 \} \times \{ 0,1 \}^{n-1}$ such that $w=1+\b^\top\y$.
    %Furthermore, if $\b$ were all $1$s, then $w=1+\sum_{j=1}^{n}y_j>m$, implying $w \notin \W(\y)$.
    %Hence, $\b$ is not all $1$s.
\end{proof}
\begin{proof}[Proof of \ref{mainthm: invsolset}\ref{mainthm: invsolset size bound}]
    To prove the set inclusion, apply Lemma \ref{lemma: bincomb}.
    Then note that $|\W(\y)|\leq |\{ 1+\b^\top\y:\b \in \{ 0 \} \times \{ 0,1 \}^{n-1} \}|=2^{n-1}$.
    % as we have $2$ choices for each of the last $n-1$ entries of $\b$.
\end{proof}
\begin{remark}
    Theorem \ref{mainthm: invsolset}\ref{mainthm: invsolset size bound} generalizes consequences of Theorems 5.1 and 5.4 in \cite{icermpaper2023inv}, which prove the result for $n=2,3$.
\end{remark}
% The result below is a strengthening of Lemma \ref{lemma: nondecreasing} for invariant parking assortments.
% It turns out that the $1+\sum_{j=1}^{n-1}y_j$ upper bound on the entries of the elements of $\PAINV_n(\y)$ is only sharp for constant $\y$; however, we can only decrease this bound by $1$ otherwise.
% \begin{theorem}
% \label{theorem: boundposcharentries}
%     Let $\y=(y_1,y_2,\dots,y_n) \in \N^n$. 
%     Assume that $\y$ is non-constant and $\chi(\y)>0$.
%     Then for any $w \in \N$ such that $(1^{n-1},w) \in \PAINV_n(\y)$, we have $w\leq \sum_{j=1}^{n-1}y_j$.
% \end{theorem}
% Lemma \ref{lemma: bincomb} also implies Theorem \ref{mainthm: invsolset size bound} that $\W(\y)$ is contained in $\{ 1+\b^\top\y: \b \in \{ 0 \} \times \{ 0,1 \}^{n-1} \}$, which is part of Theorem \ref{mainthm: invsolset}\ref{mainthm: invsolset size bound}.
% This also motivates the following definition.
% \begin{definition}
%     Let $\y \in \N^n$. The \emph{set of essential $\y$-sums} is given by
%     \[ \mathcal{S}(\y)\coloneqq \{ \b^\top\y:\b \in \{ 0 \} \times \{ 0,1 \}^{n-1} \}. \]
%     In particular, $\W(\y)\subseteq 1+\mathcal{S}(\y)$.
% \end{definition}
For brevity, write $\mathcal{S}(\y)\coloneqq \{ \b^\top\y:\b \in \{ 0 \} \times \{ 0,1 \}^{n-1} \}$, so that $\W(\y)\subseteq 1+\mathcal{S}(\y)$.
Recall that $\W(\y)$ has the largest size for certain strictly decreasing (or nonincreasing) $\y$ when $n=3$ (Theorem 5.4 in ~\cite{icermpaper2023inv}).
Thus, a natural question to ask is whether this still holds for large $n$.
The answer is yes; however, we must note that not all strictly decreasing $\y$ yield an optimal $|\W(\y)|$; one must impose a sum avoidance condition to avoid multiple $\b \in \{ 0 \} \times \{ 0,1 \}^{n-1}$ yielding the same element of $\mathcal{S}(\y)$. 
One possible condition is for $\y$ to be ``superdecreasing," which is condition (\ref{eqn: supdec}) in Theorem \ref{mainthm: invsolset}\ref{mainthm: invsolset supdec}, which states that such $\y$ indeed yield an optimal $|\W(\y)|$.
% We will now compute $\PAINV_n(\y)$ for superdecreasing $\y$ to illustrate this.
We now prove Theorem \ref{mainthm: invsolset}\ref{mainthm: invsolset supdec}. 
% the rest of Theorem \ref{mainthm: invsolset}\ref{mainthm: invsolset size bound}.
% \begin{theorem}
% \label{theorem: supdecPAINV}
%     Let $\y=(y_1,y_2,\dots,y_n) \in \N^n$ satisfy
%     \begin{equation}
%     \label{eqn: supdec}
%         y_j>\sum_{i=j+1}^{n}y_i \quad \forall j \in [n-1]
%     \end{equation} 
%     Then
%     \[ \PAINVND_n(\y)=\{ (1^{n-1},1+s):s \in \mathcal{S}(\y) \} \]
% \end{theorem}
\begin{proof}[Proof of Theorem \ref{mainthm: invsolset}\ref{mainthm: invsolset supdec}]
    We first prove that equality is achieved.
    The case for $n=1$ is immediate, so assume $n\geq 2$.
    To begin, we will first prove a convenient uniqueness property for binary combinations given $\y$ satisfying (\ref{eqn: supdec}).
    \begin{claim}
    \label{claim: bincombisuniq}
        If $s \in \mathcal{S}(\y)$, then there is a unique $\b \in \{ 0 \} \times \{ 0,1 \}^{n-1}$ such that $s=\b^\top\y$.
    \end{claim}
    \begin{proof}
        Assume for the sake of contradiction that there exist $\b,\mathbf{c} \in \{ 0 \} \times \{ 0,1 \}^{n-1}$ such that $s=\b^\top\y=\mathbf{c}^\top\y$ and $\b\neq \mathbf{c}$. 
        Write $\b=(b_1,b_2,\dots,b_n)$ and $\mathbf{c}=(c_1,c_2,\dots,c_n)$, and let
        \[ k=\min\{ i \in [n]:b_i\neq c_i \}. \]
        Without loss of generality, assume that $b_k=1$, so that $c_k=0$ (the same argument holds otherwise). 
        Then by definition,
        \[ \b^\top\y=\left(\sum_{j=1}^{k-1}b_jy_j \right)+y_k+\left(\sum_{j=k+1}^{n}b_jy_j \right)=\left(\sum_{j=1}^{k-1}c_jy_j \right)+\left(\sum_{j=k+1}^{n}c_jy_j \right)=\mathbf{c}^\top\y, \]
        so $y_k+\sum_{j=k+1}^{n}b_jy_j=\sum_{j=k+1}^{n}c_jy_j$.
        But by (\ref{eqn: supdec}), we have
        \[ \sum_{j=k+1}^{n}c_jy_j\leq \sum_{j=k+1}^{n}y_j<y_k\leq y_k+\sum_{j=k+1}^{n}b_jy_j, \]
        so $y_k+\sum_{j=k+1}^{n}b_jy_j>\sum_{j=k+1}^{n}c_jy_j$, a contradiction.
    \end{proof}
    
    %Now, we note that if $1+s \in \W(\y)$, then $s\leq \sum_{j=2}^{n}y_j$. 
    %If $s>\sum_{j=2}^{n}y_j$, then $(1+s,1^{n-1}) \notin \PA_n(\y)$, as the first car, with length $y_1$, must occupy $[1+s,s+y_1]$, and $s+y_1\geq 1+\sum_{j=1}^{n}y_j>m$. 
    %Thus, we have $s<y_1$, implying that $b_1=0$, and so by Lemma \ref{lemma: bincomb}, $s \in \mathcal{S}$; in other words, $\W(\y)\subseteq 1+\mathcal{S}$.

    We automatically have $\W(\y) \subseteq 1+\mathcal{S}(\y)$ due to Lemma \ref{lemma: bincomb}, so we show the reverse inclusion. 
    Let $s=\b^\top\y \in \mathcal{S}(\y)$ and $\x=(1^{n-1},1+s)$. 
    Consider any permutation $\x'=(x'_1,x'_2,\dots,x'_n)$ of $\x$ such that $x'_1\neq 1+s$. 
    Under the parking experiment for $\x'$, the first car occupies $[1,y_1]$, and $|[1,y_1]|=y_1>\sum_{j=2}^{n}y_j\geq \sum_{j=1}^{n}b_jy_j=s$. 
    Moreover, there exists $i \in [n]$ such that $x'_i=1+s$, and $\x'_{\vert_{i-1}}=(1^{i-1}) \in \PA_{i-1}(\y_{\vert_{i-1}})$. 
    Thus, by Lemma \ref{lemma: extendPA}, $\x'_{\vert_i} \in \PA_{i}(\y_{\vert_i})$. The remaining cars all have preference 1, so parking succeeds and $\x' \in \PA_n(\y)$.
    Now, consider the remaining permutation $(1+s,1^{n-1})$. 
    Under its parking experiment, the first car leaves $[1,s]$ empty. 
    Claim \ref{claim: bincombisuniq} guarantees that a unique subset of the remaining $n-1$ cars, say cars $i_1,i_2,\dots,i_q$, can precisely fill $[1,s]$, so $\sum_{j=1}^{q} y_{i_j}=s$.
    \begin{claim}
    \label{claim: correctcarspark}
        Cars $i_1,i_2,\dots,i_q$ are the only ones that will park in $[1,s]$ under the parking experiment for $(1+s,1^{n-1})$.
    \end{claim}
    \begin{proof}
        Due to (\ref{eqn: supdec}), 
        \[ y_2>y_3>\dots>y_{i_1-1}>\sum_{j=i_1}^{n}y_j\geq\sum_{j=1}^{q}y_{i_j}=s, \]
        so cars $2,3,\dots,i_1-1$, all with preference $1$, will drive past $[1,s]$ and successively park immediately after.
        Thus, car $i_1$ must park in $[1,s]$, establishing the base case of our claim.
        Now, inductively, assume that cars $i_1,i_2,\dots,i_p$ are the only ones that parked in $[1,s]$ among the first $i_p$ cars, where $p \in [q-1]$. 
        Note that the unoccupied spots of $[1,s]$ are $[\sum_{j=1}^{p}y_{i_j},s]$, which has length $\sum_{j=p+1}^{q}y_{i_j}$. 
        Then by (\ref{eqn: supdec}), we have
        \[ y_{i_p+1}>y_{i_p+2}>\dots>y_{i_{p+1}-1}>\sum_{j=i_{p+1}}^{n}y_j\geq\sum_{j=p+1}^{q}y_{i_j}, \]
        so cars $i_p+1,i_p+2,\dots,i_{p+1}-1$ will drive past $[\sum_{j=1}^{p}y_{i_j},s]$ (and successively park immediately after). 
        Thus, car $i_{p+1}$ is the next car to park in $[1,s]$, completing the induction.
    \end{proof}

    From Claim \ref{claim: correctcarspark}, it follows that parking succeeds since each car $j$, where $j \in [n]\setminus\{ 1,i_1,i_2,\dots,i_q \}$, has preference $1$ and drives past $[1,s]$, so they will successively fill $[1+s+y_1,m]$. 
    Hence, we have $(1^{n-1},1+s) \in \PA_n(\y)$, so $\x \in \PAINV_n(\y)$ and $\W(\y)=1+\mathcal{S}(\y)$.

    To conclude, we will prove the following.
    
    \begin{claim}
        If $\y$ satisfies (\ref{eqn: supdec}), then $\chi(\y)=1$.
    \end{claim}   
    \begin{proof}
        We first prove that there cannot exist $1<w_1\leq w_2$ such that $(1^{n-2},w_1,w_2) \in \PAINV_n(\y)$. 
        Assume otherwise. 
        Then Lemma \ref{lemma: invsolset}, Lemma \ref{lemma: bincomb}, and our result above that $\W(\y)=1+\mathcal{S}(\y)$ guarantee $w_1-1=\b^\top\y$ for some $\b \in \{ 0 \} \times \{ 0,1 \}^{n-1}$.
        Now, let
        \[ k=\min\{ i \in [n]:b_i=1 \} \]
        (we note that $k>1$ since $b_1=0$).
        Consider the preferences $(w_1,1^{k-2},w_2,1^{n-k})$. 
        Under its parking experiment, the first car leaves $[1,\b^\top\y]$ empty, where $|[1,\b^\top\y]|=\b^\top\y$.
        By Claim \ref{claim: bincombisuniq}, $\b$ is unique, so to ensure $[1,\b^\top\y]$ is completely filled, car $k$ must park here.
        %there exists only one index set $I\subseteq [n]\setminus \{ 1 \}$ such that $\sum_{i \in I}y_i=\b^\top\y$, where $k \in I$. 
        However, by construction, this is not the case since car $k$ has preference $w_2\geq w_1=1+\b^\top\y$. 
        Hence, $[1,\b^\top\y]$ will have unoccupied spaces by the end of the parking experiment, so parking fails, and it follows that if $\x \in \PAINV_n(\y)$, then $\deg \x\neq 2$.  
        Thus, $\chi(\y)<2$ by Theorem \ref{mainthm: degchar}\ref{mainthm: degchar image}, and so $\chi(\y)=1$ since $\W(\y)\supsetneq \{ 
        1 \}$.
    \end{proof}
    
    Therefore, $\PAINVND_n(\y)=\{ (1^{n-1},1+s):s \in \mathcal{S}(\y) \}$, as desired.

    For our count, note that the only nondecreasing invariant parking assortments of $\y$ are of the form $(1^{n-1},1+s)$, where $s=\b^\top\y$ for some $\b \in \{ 0 \} \times \{ 0,1 \}^{n-1}$. 
    Claim \ref{claim: bincombisuniq} ensures that $\b$ is unique, so each choice of $\b$ gives a distinct value of $s$. 
    Thus, by definition, $|\PAINVND_{n}(\y)|=|\mathcal{W}(\y)|=2^{n-1}$.
    
    Lastly, note that each of the $2^{n-1}-1$ nontrivial elements of $\PAINVND_{n}(\y)$ have $n$ distinct permutations of their entries due to their unique non-$1$ entry. 
    Therefore, $|\PAINV_n(\y)|=n(2^{n-1}-1)+1=2^{n-1}n-n+1$, as claimed.
\end{proof}
% \begin{corollary}
% \label{corollary: strictdeccount}
%     Let $\y=(y_1,y_2,\dots,y_n) \in \N^n$ satisfy (\ref{eqn: supdec}). 
%     Then
%     \[ |\PAINV_n(\y)|=2^{n-1}n-n+1 \quad \text{and} \quad |\PAINVND_{n}(\y)|=|\mathcal{W}(\y)|=2^{n-1}. \]
% \end{corollary}
% \begin{proof}
%     By Theorem \ref{theorem: supdecPAINV}, the only nondecreasing invariant parking assortments of $\y$ are of the form $(1^{n-1},1+s)$, where $s=\b^\top\y$ for some $\b \in \{ 0 \} \times \{ 0,1 \}^{n-1}$. 
%     Claim \ref{claim: bincombisuniq} ensures that $\b$ is unique, so each choice of $\b$ gives a distinct value of $s$. 
%     Thus, by definition, $|\PAINVND_{n}(\y)|=|\mathcal{W}(\y)|=2^{n-1}$.
    
%     Lastly, note that each of the $2^{n-1}-1$ nontrivial elements of $\PAINVND_{n}(\y)$ have $n$ distinct permutations of their entries due to their unique non-$1$ entry. 
%     Therefore, $|\PAINV_n(\y)|=n(2^{n-1}-1)+1=2^{n-1}n-n+1$, as claimed.
% \end{proof}
We are now ready to prove the rest of Theorem \ref{mainthm: invsolset}, which gives a necessary condition for equality to be achieved in Theorem \ref{mainthm: invsolset}\ref{mainthm: invsolset size bound}.
\begin{lemma}
\label{lemma: sufffirstentrymax}
    Let $\y=(y_1,y_2,\dots,y_n) \in \N^n$.
    If $1+\sum_{j=2}^{n}y_j \in \W(\y)$, then $y_1=\max\y$.
\end{lemma}
\begin{proof}
    If $y_1<\max\y$, let $k \in [n] \setminus \{ 1 \}$ satisfy $y_k=\max \y$, and consider the preferences $(1^{k-1},1+\sum_{j=2}^{n}y_j,1^{n-k})$. 
    Then under its parking experiment, the first $k-1$ cars park successively to fill $[1,\sum_{j=1}^{k-1}y_j]$.
    Car $k$ occupies $[s,s+y_k-1]$, where $s\geq 1+\sum_{j=2}^{n}y_j$, so $s+y_k-1\geq y_k+\sum_{j=2}^{n}y_j>m$ and hence parking fails. 
    Thus, $(1^{n-1},1+\sum_{j=2}^{n}y_j) \notin \PAINV_n(\y)$, so we have the contrapositive of the desired result.
\end{proof}
% \begin{theorem}
%     Let $\y=(y_1,y_2,\dots,y_n) \in \N^n$. 
%     Then $|\W(\y)|\leq 2^{n-1}$. 
%     Moreover, if equality holds, then we must have
%     \begin{equation}
%     \label{eqn: almostsupdec}
%         y_1\geq y_2 \quad \text{and} \quad y_j>\sum_{i=j+1}^{n}y_i \quad \forall j \in [n-1] \setminus \{ 1 \}.
%     \end{equation}
% \end{theorem}
\begin{proof}[Proof of Theorem \ref{mainthm: invsolset}\ref{mainthm: invsolset neccequality}]
    % Applying Lemma \ref{lemma: bincomb}, we have $\W(\y)\subseteq 1+\mathcal{S}(\y)$, so $|\W(\y)|\leq |\mathcal{S}(\y)|$. 
    % Each element of $\mathcal{S}(\y)$ is determined by $\b \in \{ 0 \} \times \{ 0,1 \}^{n-1}$, of which there are $2^{n-1}$ choices. 
    % Thus, $|\W(\y)|\leq |\mathcal{S}(\y)|\leq 2^{n-1}$. 
    % Theorem \ref{theorem: supdecPAINV} guarantees that equality is achieved when $\y$ satisfies (\ref{eqn: supdec}).

    If $|\W(\y)|=2^{n-1}$, then $|\W(\y)|=|\mathcal{S}(\y)|=2^{n-1}$. 
    Thus, there cannot exist $\b,\mathbf{c} \in \{ 0 \} \times \{ 0,1 \}^{n-1}$ such that $\b\neq \mathbf{c}$ and $\b^\top\y=\mathbf{c}^\top\y$; otherwise, $|\mathcal{S}(\y)|\leq 2^{n-1}-1$.
    Then $1+\b^\top\y \in \W(\y)$ for any $\b \in \{ 0 \} \times \{ 0,1 \}^{n-1}$.
    In particular, choosing $\b=(0,1^{n-1})$ gives $1+\sum_{j=2}^{n}y_j \in \W(\y)$.
    Lemma \ref{lemma: sufffirstentrymax} then ensures $y_1=\max\y$, so that $y_1\geq y_2$.

    We now show that the remaining inequality detailed in (\ref{eqn: almostsupdec}) holds via induction.
    For the base case, suppose $y_2\leq \sum_{i=3}^{n}y_i$.
    For $\b=(0^2,1^{n-2})$, we have $1+\sum_{i=3}^{n}y_i \in \W(\y)$.
    Hence, $(1+\sum_{i=3}^{n}y_i,1^{n-1}) \in \PA_n(\y)$, so under its parking experiment, the first car leaves $[1,\sum_{i=3}^{n}y_i]$ empty for other cars to fill.
    As $y_2\leq \sum_{i=3}^{n}y_i$, car 2 parks in $[1,\sum_{i=3}^{n}y_i]$.
    Thus, there exists $\mathbf{c}=(c_1,c_2,\dots,c_n) \in \{ 0 \} \times \{ 0,1 \}^{n-1}$ with $c_2=1$ such that $\mathbf{c}^\top\y=\sum_{i=3}^{n}y_i=\b^\top\y$, a contradiction since $\b\neq\mathbf{c}$.
    This shows that $y_2>\sum_{i=3}^{n}y_i$.

    Now, assume that for some $p \in [n] \setminus \{ 1 \}$, we have $y_j>\sum_{i=j+1}^{n}y_i$ for all $j \in [p-1] \setminus \{ 1 \}$. 
    We seek to prove that $y_p>\sum_{i=p+1}^{n}y_i$. 
    To do so, we will employ the same argument from above.
    Assume $y_p\leq \sum_{i=p+1}^{n}y_i$. 
    Choosing $\b=(0^p,1^{n-p})$
    %, we have $1+\sum_{i=p+1}^{n}y_i \in \W(\y)$. 
    yields $(1+\sum_{i=p+1}^{n}y_i,1^{n-1}) \in \PA_n(\y)$, so under its parking experiment, $[1,\sum_{i=p+1}^{n}y_i]$ is filled by a subset of cars $2,3,\dots,n$.
    By the inductive hypothesis, for any $j \in [p-1] \setminus \{ 1 \}$,
    \[ y_j>\sum_{i=j+1}^{n}y_i>\sum_{i=p+1}^{n}y_i. \]
    Since $y_p\leq \sum_{i=3}^{n}y_i$, car $p$ must be the first to park in $[1,\sum_{i=p+1}^{n}y_i]$.
    Thus, there exists $\mathbf{c}=(c_1,c_2,\dots,c_n) \in \{ 0 \} \times \{ 0,1 \}^{n-1}$ with $c_p=1$ such that $\mathbf{c}^\top\y=\sum_{i=p+1}^{n}y_i=\b^\top\y$, so we arrive at another contradiction.
    Therefore, $y_p>\sum_{i=p+1}^{n}y_i$, completing the induction.
\end{proof}
\begin{remark}
    The condition (\ref{eqn: almostsupdec}) is not sufficient for $|\W(\y)|=2^{n-1}$. 
    For $\y=(7,5,3,1)$, one can check that $\W(\y)=\{ 1,2,4,5,6,7,9 \}$, so $|\W(\y)|=7<2^3$.
\end{remark}
Given what was discussed above, we conclude by proving Theorem \ref{mainthm: PAINVNDbound}, which gives an upper bound on $|\PAINVND_n(\y)|$ independent of $m$.
% \begin{corollary}
% \label{corollary: estimate}
%     Let $\y \in \N^n$.
%     Then $\PAINVND_n(\y) \subseteq \{ 1 \}^{n-\chi(\y)} \times \W(\y)^{\chi(\y)}\subseteq \{ 1 \}^{n-\chi(\y)} \times (1+\mathcal{S}(\y))^{\chi(\y)}$.
%     In particular, $|\PAINVND_n(\y)|\leq \binom{2^{n-1}+n-2}{n-1}$.
% \end{corollary}
\begin{proof}[Proof of Theorem \ref{mainthm: PAINVNDbound}]
    The set inclusion $\PAINVND_n(\y) \subseteq \{ 1 \}^{n-\chi(\y)}\times \W(\y)^{\chi(\y)}$ follows by Definition \ref{defn: degandchar} and Lemma \ref{lemma: invsolset}.
    To prove the bound, note that the nondecreasing elements of $\W(\y)^{\chi(\y)}$ are precisely multisets of cardinality $\chi(\y)$ taken from $\W(\y)$.
    Letting $|\W(\y)|=W$ and $\chi(\y)=C$, by stars and bars, the number of such multisets is $B(W,C)\coloneqq \binom{W+C-1}{C}$.
    By Pascal's identity, $B(W+1,C)=\binom{(W+1)+C-1}{C}=\binom{(W+1)+C-2}{C-1}+\binom{(W+1)+C-2}{C}=\binom{W+C-1}{C}+\binom{(W+1)+C-2}{C-1}\geq B(W,C)$, so $B(W,C)$ is nondecreasing with respect to $W$. 
    Similarly, $B(W,C+1)=\binom{W+(C+1)-1}{C+1}=\binom{W+C-1}{C}+\binom{W+C-1}{C+1}\geq B(W,C)$, so $B(W,C)$ is nondecreasing with respect to $C$.
    Therefore, since $W\leq 2^{n-1}$ by Theorem \ref{mainthm: invsolset}\ref{mainthm: invsolset size bound} and $C\leq n-1$, we have $|\PAINVND_n(\y)|\leq B(W,C)\leq B(2^{n-1},n-1)=\binom{2^{n-1}+n-2}{n-1}$, as claimed.
    % which has size at most $2^{n-1}$.
    % Since $\chi(\y)\leq n-1$, by stars and bars, the number of such multisets is at most $\binom{2^{n-1}+n-2}{n-1}$, which implies the bound.
\end{proof}
\section{Open Problems}
\label{section: openproblems}
We now suggest various problems for future research.
\subsection{Length Vectors of Non-Maximal Characteristic}
In Section \ref{section: almostconstant}, we found a simple closed form for the set $\{ \y \in \N^n:\chi(\y)=n-1 \}$.
A natural follow-up to this is to consider finding closed forms for other characteristics.
\begin{openproblem}
\label{openproblem: chialphan}
    Give a direct characterization of the preimage $\chi^{-1}_n(\alpha)\coloneqq \{ \y \in \N^n:\chi(\y)=\alpha \}$ for any $\alpha \in [n-1]_0$.
\end{openproblem}
We note that Theorem \ref{mainthm: neccminchar} gives a containment of $\chi^{-1}_n(0)$ in a relatively simple set.
However, refining this containment is the main difficulty. 
For any $\alpha$, we suspect that an effective way to approach the problem is to look at repeated entries of $\y$ or any matching partial sums of the entries of $\y$ and study how these factors affect $\chi(\y)$. 
\subsection{Preserving the Characteristic}
As discussed, Theorem \ref{mainthm: degchar}\ref{mainthm: degchar prefix} gives an easy way to bound $\chi(\y)$ given that we know $\chi(\y_{\vert_k})$ for some $k<n$.
This can simplify the process for computing $\PAINV_n(\y)$ given $\PAINV_n(\y_{\vert_k})$, but there are still a great deal of possibilities to check especially if $k\ll n$.
In light of this case, we ask the following.
\begin{openproblem}
\label{openproblem: charpreserve}
    Let $\y \in \N^n$, where $\chi(\y)=\alpha$, and $\y^+=(\y,y_{n+1}) \in \N^{n+1}$. 
    What must be true about $\y$ and $\y^+$ to guarantee that $\chi(\y^+)=\alpha$?
\end{openproblem}
Note that if $\alpha=0$, then an answer to the above can help one construct $\chi(0,n)$ using a recursive-like technique.
Moreover, given a condition on $\y$ and $\y^+$, one tool that may be useful in showing that $\chi(\y^+)\neq n+1$ is Theorem \ref{mainthm: invsolset}\ref{mainthm: invsolset nonconstant}.
\subsection{Sharper Bounds}
Recall that in Theorem \ref{mainthm: PAINVNDbound}, we had an upper bound for the size of $\PAINVND_n(\y)$ depending only on $n$.
One might ask if this is the best such bound or if there is a similar bound for the size of $\PAINV_n(\y)$.
For our bound, we posit that the answer is no, and our computational experiments suggest that there is a familiar connection between the best upper bound for $|\PAINV_n(\y)|$ and the best upper bound for $|\PAINVND_n(\y)|$, which is described below.
\begin{openproblem}
    Let $\y \in \N^n$. 
    Do $|\PAINV_n(\y)|\leq (n+1)^{n-1}$ and $|\PAINVND_n(\y)|\leq C_n$ hold for any $n$?
\end{openproblem}
If this is true, then the upper bounds are automatically sharp since constant $\y$ are examples of the equality case.
A potential way to approach this problem is to examine and understand which elements of $1+\mathcal{S}(\y)$ also are in $\W(\y)$, as well as apply any results for Open Problems \ref{openproblem: chialphan} and \ref{openproblem: charpreserve}.

\section*{Acknowledgements}
% This material is based on an independent continuation of past work ~\cite{icermpaper2023inv} supported by the Institute for Computational and Experimental Research in Mathematics.
% I would like to thank Carlos Mart\'{i}nez for fruitful discussions on the material here and reviewing a draft of this work.

% I would also like to thank Pamela Harris for suggesting the problem of invariant parking and providing various directions for research.
I would like to thank Carlos Mart\'{i}nez and Pamela Harris for fruitful discussions.
\printbibliography
\end{document}